\documentclass[12pt,reqno]{amsart}
\usepackage[T1]{fontenc}
\usepackage{lmodern}
\usepackage[colorlinks=true]{hyperref}

\usepackage{amsxtra,mathrsfs,mathtools,thmtools,amssymb}

\usepackage{mathptmx}
\usepackage{MnSymbol}

\usepackage{microtype}
\DisableLigatures[f]{encoding = *, family = * }

\usepackage{tikz,tikz-cd}
\usepackage{caption}
\usepackage{enumitem}
\setlist[enumerate,1]{font=\upshape}

\allowdisplaybreaks[4]

\usepackage[capitalise]{cleveref}

\newtheorem{theorem}{Theorem}[section]
\newtheorem{proposition}[theorem]{Proposition}
\newtheorem{lemma}[theorem]{Lemma}
\newtheorem{corollary}[theorem]{Corollary}

\newtheorem{example}[theorem]{Example}

\theoremstyle{remark}
\newtheorem{remark}[theorem]{Remark}

\theoremstyle{definition}
\newtheorem{definition}[theorem]{Definition}

\numberwithin{equation}{section}

\begin{document}
\title[Topological entropy dimension]{Topological entropy dimension on subsets for nonautonomous dynamical systems}

\author[Chang-Bing Li]{Chang-Bing Li}

\address{Chang-Bing Li: Department of Mathematics, Shantou University, Shantou, Guangdong 515821, P.R. China}

\email{christiesyp@gmail.com}

\date{\today}



\begin{abstract}
The topological entropy dimension is mainly used to distinguish the zero topological entropy systems. Two types of topological entropy dimensions, the classical entropy dimension and the Pesin entropy dimension, are investigated for nonautonomous dynamical systems. Several properties of the entropy dimensions are discussed, such as the power rule, monotonicity and equiconjugacy et al. The Pesin entropy dimension is also proved to be invariant up to equiconjugacy. The relationship between these two types of entropy dimension is also discussed in more detail. It's proved that these two entropy dimensions coincide and are equal to one provided that the classical topological entropy is positive and finite.
\end{abstract}

\maketitle

\section{Introduction}
The topological entropy, as an important topological conjugate invariant, was firstly introduced by Adler et al. ~\cite{AKM1965}. Bowen ~\cite{Bow1973} used a method resembles the Hausdorff dimension to extend the concept of topological entropy for noncompact sets. Pesin ~\cite{PP1984,Pes1997} developed the so-called Carath{\'e}odory-Pesin structure to study the topological entropy and topological pressure on noncompact sets, and proved the equivalence between Bowen entropy and Pesin entropy for compact systems. It is well known that the topological entropy measures the exponential growth rate of the number of distinguishable orbits as time advances and takes value in $[0,+\infty]$. The \emph{positive entropy systems} refer to dynamical systems with positive topological entropy, while \emph{infinite entropy systems} are special case of them, where the topological entropy reach infinity. The present paper concerns mainly on the finite entropy systems, especially the \emph{zero entropy systems} or \emph{null systems}, for which the topological entropy are zero. Although the positive entropy systems exhibit more dynamical complexity, the zero entropy systems own various levels of complexity, and recently have been discussed by many studies ~\cite{KCML2014}.  

The topological entropy dimension is the main tool to distinguish the zero entropy systems, which was firstly introduced by Carvalho ~\cite{Car1997} to classify different levels of intermediate growth rate of complexity for dynamical systems ~\cite{DP2022}. In practical terms, the topological entropy dimension can provide insight into the behavior of complex systems, especially those that exhibit fractal characteristics. A typical dynamical system with positive and finite topological entropy has topological entropy dimension $1$. A higher entropy dimension indicates a higher degree of complexity or randomness in the measure's distribution in dynamical systems. The inherent properties of entropy dimension have led to numerous applications, for examples, Dou et al. \cite{DHP2011} introduced the concept of strictly positive entropy dimension and showed that if a dynamical system has strictly positive entropy dimension, then it is weakly mixing, and Zeng \cite{zeng2024} further proved that the weakly mixing property implies Kato chaos. Meanwhile, the entropy dimension is closed related to the sequence entropy pair, discrete spectrum and minimal systems, and which are beyond the scope of the present paper. For a more comprehensive understanding of these concepts, we would like to refer the reader to Chapter 7 of the monograph of Ye, Huang and Shao \cite{YHS2008}. Throughout this paper, our focus remains on the topological properties of entropy dimension, therefore, we may occasionally refer to the topological entropy dimension as entropy dimension for convenience. 

The entropy dimension has been widely studied by many researchers, on both the autonomous systems and the nonautonomous systems. Ferenczi and Park ~\cite{FP2007} introduced the entropy dimension to measure the complexity of entropy zero measurable dynamics. Cheng and Li ~\cite{CL2010} showed that the entropy dimension takes on arbitrary values in $[0,1]$. Ahn et al. ~\cite{ADP2010,DHP2011,DP2022} did some deep investigations and showed that the entropy dimension doesn't have the variational principle, they also showed that the topological entropy dimension can be given by the generating sequence. 
Ma et al. ~\cite{MKL2012,KCML2014} discussed the entropy dimension on noncompact sets, and compared different forms of topological entropy dimensions. Recently, the relative entropy dimension, preimage entropy dimension for autonomous dynamical systems and the topological entropy dimension for amenable group actions were also discussed ~\cite{Zho2019,LZZ2021,LJZ2025}.

After the study of topological entropy for nonautonomous dynamical system by Kolyada and Snoha ~\cite{KS1996}. Kuang et al. ~\cite{KCL2013} discussed some properties of the topological entropy dimension, such as the power rule, monotonicity, equiconjugacy for upper topological entropy dimension. Li et al. ~\cite{LZW2019} constructed concrete examples of zero entropy systems with entropy dimension attain any value in $[0,1]$. Yang et al. ~\cite{YWW2020} further discussed the measure-theoretic entropy dimension. However, the studies of entropy dimension are far from satisfactory. 

Inspired by several types of topological entropies in nonautonomous dynamical systems ~\cite{li2015,JY2021}, the purpose of this paper is to define and study the classical topological entropy dimension and the Pesin topological entropy dimension, and to further investigate the relationship between several types of topological entropy dimensions and the corresponding topological entropy in nonautonomous dynamical systems.

This paper is organized as follows. In the second section, we give some basic concepts and results for the classical topological entropy dimension and extend the concept of Pesin topological entropy to Pesin topological entropy dimension on subsets for nonautonomous dynamical systems. In the third section, we study the properties of the classical topological entropy dimension that are not covered by the known literature, and investigate the properties of the Pesin topological entropy dimension in the meantime. The results show that the Pesin topological entropy dimension shares many similar properties with the classical topological entropy dimension, such as the monotonicity, power rule, equiconjugacy, et al. In the last section, we investigate the relationship among the two types of classical topological entropy dimension and the Pesin topological entropy dimension, and discuss the relationship between these two topological entropy dimensions and the corresponding topological entropies. We further prove that for the positive topological entropy systems, the two types of classical entropy dimension and the Pesin entropy dimension coincide and are all equal to one.

\section{Preliminaries}
Let $X$ be a compact topological space, and $f_{1,\infty}=\{f_i\}_{i=1}^{\infty}$ be a sequence of continuous maps from $X$ to itself. For any $i\in\mathbb{N}$, let $f_i^0=\textnormal{id}_X$ be the identity map of $X$, and for any $i,j\in\mathbb{N}$, let 
\[
f_i^j=f_{i+(j-1)}\circ\dots\circ f_{i+1}\circ f_i,\quad f_i^{-j}=\left(f_i^j\right)^{-1}=f_i^{-1}\circ f_{i+1}^{-1}\circ\dots\circ f_{i+(j-1)}^{-1}.
\]
Then we call such system a \emph{nonautonomous dynamical system} (NDS for short), and denoted it by $(X,f_{1,\infty})$. For any $x\in X$, the \emph{trajectory} and the \emph{orbit} of $x$ are the sequence $(f_1^n(x))_{n=0}^{\infty}$ and the set $\{f_1^n(x): n=0,1,\cdots\}$  ~\cite{KS1996}. Note that, if $f_i=f$ for every $i\geq 1$, we get the usual topological dynamical system $(X,f)$.

Let $(X,f_{1,\infty})$ be an NDS. We denote by $f_{1,\infty}^k=\{f_{ik+1}^k\}_{i=0}^{\infty}=\{f_1^k,f_{k+1}^k,f_{2k+1}^k,\dots\}$, by $f_{1,\infty}^{-1}=\{f_i^{-1}\}_{i=1}^{\infty}$ and by $f_{i,\infty}=\{f_i,f_{i+1},\dots\}$. For any finite open covers $\mathcal{U}_1,\mathcal{U}_2,\dots,\mathcal{U}_n$ of $X$, we denote by
\[
\bigvee_{i=1}^n\mathcal{U}_i=\mathcal{U}_1\vee\mathcal{U}_2\vee\dots\vee\mathcal{U}_n=\left\{U_1\cap U_2\cap\dots\cap U_n:U_i\in\mathcal{U}_i,i=1,2,\dots, n\right\}.
\]
For any finite open cover $\mathcal{U}$ of $X$, we also denote by $f_i^{-j}(\mathcal{U})=\left\{f_i^{-j}(U):U\in\mathcal{U}\right\}$, by $\mathcal{U}_i^n=\bigvee_{j=0}^{n-1}f_i^{-j}(\mathcal{U})$, and by $\mathcal{N}(\mathcal{U})$ the minimal possible cardinality of a subcover chosen from $\mathcal{U}$. Let $K$ be a nonempty subset of $X$ (not necessarily compact), we denote by $\mathcal{U}|_K$ the cover $\{U\cap K:U\in\mathcal{U}\}$ of $K$. Given two covers $\mathcal{U},\mathcal{V}$ of $X$, we say $\mathcal{V}$ is a \emph{refinement} of $\mathcal{U}$, if every member of $\mathcal{V}$ is included in a member of $\mathcal{U}$, and denote it by $\mathcal{U}\preceq\mathcal{V}$. 

The following \Cref{L:useful} and \Cref{onK} are simple and useful in this paper, so we present them and skip the proof.

\begin{lemma}\label{L:useful}
Let $(X,f_{1,\infty})$ be an NDS. Then for any open covers $\mathcal{U},\mathcal{V}$ of $X$,
\begin{enumerate}
\item $\mathcal{U}\preceq\mathcal{U}\vee\mathcal{V};$

\item $\mathcal{N}(\mathcal{U}\vee\mathcal{V})\leq \mathcal{N}(\mathcal{U})\cdot\mathcal{N}(\mathcal{V});$

\item $f_i^{-1}(\mathcal{U}\vee\mathcal{V})=f_i^{-1}(\mathcal{U})\vee f_i^{-1}(\mathcal{V})$ for any $i\geq 1;$
	
\item $\mathcal{N}(f_i^{-j}(\mathcal{U}))\leq \mathcal{N}(\mathcal{U})$ for any $i\geq 1$ and $j\geq 0;$

\item if $\mathcal{U}\preceq\mathcal{V}$ then $\mathcal{N}(\mathcal{U})\leq \mathcal{N}(\mathcal{V})$ and $\mathcal{U}_i^n\preceq\mathcal{V}_i^n$ for any $i,n\geq 1;$

\item if $Y\subseteq Z\subseteq X$ then $\mathcal{N}(\mathcal{U}|_Y)\leq \mathcal{N}(\mathcal{U}|_Z)$.
\end{enumerate}
\end{lemma}

\begin{lemma}\label{onK}
Let $(X,f_{1,\infty})$ be an NDS, and $K$ be a nonempty subset of $X$. Then for any open covers $\mathcal{U},\mathcal{V}$ of $X$,
\begin{enumerate}
\item $\mathcal{N}(\mathcal{U}|_K)\leq \mathcal{N}(\mathcal{U});$

\item $\left(\mathcal{U}\vee \mathcal{V}\right)|_K=(\mathcal{U}|_K)\vee (\mathcal{V}|_K);$

\item $f_i^{-1}(\mathcal{U})|_K=\left(f_i|_K\right)^{-1}(\mathcal{U}|_K)$ for any $i\geq 1;$

\item if $\mathcal{U}\preceq \mathcal{V}$ then $\mathcal{N}(\mathcal{U}|_K)\leq \mathcal{N}(\mathcal{V}|_K)$.
\end{enumerate}
\end{lemma}

\subsection{Classical topological entropy dimension}
\begin{definition}\cite{KS1996}
Let $(X,f_{1,\infty})$ be an NDS, and $K$ be a nonempty subset of $X$. The \emph{topological entropy of the sequence of maps $f_{1,\infty}$ on the set $K$} is defined by 
\[
h_{top}(f_{1,\infty},K)=\sup\left\{h_{top}(f_{1,\infty},K,\mathcal{U}):\mathcal{U} \text{ is an open cover of } X\right\},
\]
where 
\begin{align}\label{entropy}
h_{top}(f_{1,\infty},K,\mathcal{U})=\mathop{\lim\sup}_{n\to\infty}\frac{1}{n}\log \left(\bigvee_{j=0}^{n-1}f_1^{-j}(\mathcal{U})\Big|_K\right) =\mathop{\lim\sup}_{n\to\infty}\frac{1}{n}\log\mathcal{N}\left(\mathcal{U}_1^n\big|_K\right).
\end{align}
Particularly, the topological entropy of $f_{1,\infty}$ on $X$ is abbreviated as $h(f_{1,\infty})$.
\end{definition}

For any $s>0$, we replace $\frac{1}{n}$ by $\frac{1}{n^s}$ in ~\eqref{entropy}, and let
\begin{align*}
\overline{h}_K(f_{1,\infty},s,\mathcal{U})&=\mathop{\lim\sup}_{n\to\infty}\frac{1}{n^s}\log\mathcal{N}\left(\mathcal{U}_1^n\big|_K\right),\\
\underline{h}_K(f_{1,\infty},s,\mathcal{U})&=\mathop{\lim\inf}_{n\to\infty}\frac{1}{n^s}\log\mathcal{N}\left(\mathcal{U}_1^n\big|_K\right).
\end{align*}

\begin{definition}
Let $(X,f_{1,\infty})$ be an NDS, and $K$ be a nonempty subset of $X$. For any $s>0$, the \emph{upper and lower $s$-topological entropy} of $f_{1,\infty}$ on the set $K$ are defined by 
\[
\overline{h}_K(f_{1,\infty},s)=\sup\{\overline{h}_K(f_{1,\infty},s,\mathcal{U})\},\quad \underline{h}_K(f_{1,\infty},s)=\sup\{\underline{h}_K(f_{1,\infty},s,\mathcal{U})\},
\]
respectively, where the supremum is taken over all finite open covers of $X$.
\end{definition}

When $s=1$, then $\overline{h}_K(f_{1,\infty},1)=h_{top}(f_{1,\infty},K)$. It is easy to see that $\underline{h}_K(f_{1,\infty},s)\leq \overline{h}_K(f_{1,\infty},s)$ and $\overline{h}_K(f_{1,\infty},s)$, $\underline{h}_K(f_{1,\infty},s)$ are non-increasing with $s>0$. Besides $\overline{h}_K(f_{1,\infty},s),\underline{h}_K(f_{1,\infty},s)\notin \{0,+\infty\}$ for at most one $s>0$. This implies that the values of $\overline{h}_K(f_{1,\infty},s)$ and $\underline{h}_K(f_{1,\infty},s)$ jump from $+\infty$ to $0$ at some critical points.

\begin{definition}\cite{LZW2019}
Let $(X,f_{1,\infty})$ be an NDS, and $K$ be a nonempty subset of $X$. The \emph{upper and lower topological entropy dimension of $f_{1,\infty}$} on $K$ are defined by 
\begin{align*}
	\overline{D}_K(f_{1,\infty})&=\inf\{s>0:\overline{h}_K(f_{1,\infty},s)=0\}=\sup\{s>0: \overline{h}_K(f_{1,\infty},s)=\infty\},\\
	\underline{D}_K(f_{1,\infty})&=\inf\{s>0:\underline{h}_K(f_{1,\infty},s)=0\}=\sup\{s>0: \underline{h}_K(f_{1,\infty},s)=\infty\},
\end{align*}
respectively. Especially when $\overline{D}_K(f_{1,\infty})=\underline{D}_K(f_{1,\infty})=s$, we call $s$ \emph{the classical topological entropy dimension} of $K$, and denoted it by $D_K(f_{1,\infty})=s$.
\end{definition}

\begin{remark}
We call the above entropy dimension \emph{the classical topological entropy dimension}, in order to distinguish the Pesin topological entropy dimension in the next subsection. It should be noticed that in ~\cite{KCL2013}, Kuang et al. called the upper entropy dimension $\overline{D}_K(f_{1,\infty})$ as the topological entropy dimension of $f_{1,\infty}$ . 
\end{remark}

We would like to mention that Dou et al. \cite{DHP2011} altered the order of the procedure, that is, first take the critical value of $s$ to define the \emph{upper and lower entropy dimension with respect to $\mathcal{U}$} by
\begin{align*}
\overline{D}_K'(f_{1,\infty},\mathcal{U})&=\inf\{s>0:\overline{h}_K(f_{1,\infty},s,\mathcal{U})=0 \}=\sup\{s>0:\overline{h}_K(f_{1,\infty},s,\mathcal{U})=\infty \};\\
\underline{D}_K'(f_{1,\infty},\mathcal{U})&=\inf\{s>0:\underline{h}_K(f_{1,\infty},s,\mathcal{U})=0 \}=\sup\{s>0:\underline{h}_K(f_{1,\infty},s,\mathcal{U})=\infty \}.
\end{align*}
Then take the supremum of all open covers of $X$:
\[
\overline{D}_K'(f_{1,\infty})=\sup_{\mathcal{U}}\left\{\overline{D}_K'(f_{1,\infty},\mathcal{U}):\mathcal{U}\right\},\quad \underline{D}_K'(f_{1,\infty})=\sup_{\mathcal{U}}\left\{\underline{D}_K'(f_{1,\infty},\mathcal{U})\right\}.
\]

\begin{proposition} \label{p:equiv}
Let $(X,f_{1,\infty})$ be an NDS, and $K$ be a nonempty subset of $X$. Then
\[
\overline{D}_K(f_{1,\infty})=\overline{D}_K'(f_{1,\infty}),\quad \underline{D}_K(f_{1,\infty})=\underline{D}_K'(f_{1,\infty}).
\]
\end{proposition}

\begin{proof}
The proof is similar to ~\cite[Proposition 3.1]{MKL2012}, so we omit it.
\end{proof}

The topological entropy dimension can also be defined in terms of separated sets and spanning sets. Let $(X,d)$ be a compact metric space. For any $A\subseteq X$, the diameter of $A$ is defined as $|A|=\max\{d(x,y):x,y\in A\}$. For any open cover $\mathcal{U}$ of $X$, the diameter of $\mathcal{U}$ is given by $|\mathcal{U}|=\max\{|A|:A\in \mathcal{U}\}$. For any $\varepsilon>0$, $n\in \mathbb{N}$, we define a new metric on $X$ by
\[
d_n(x,y)=\max\left\{d\left(f_1^j(x)),f_1^j(y)\right):0\le j\leq n-1 \right\}.
\]
It is easy to check that $d_n$ is also a metric of $X$, and is equivalent to $d$ since $X$ is compact. The Bowen ball is given by $B_n(x,\varepsilon)=\{y\in X:d_n(x,y)<\varepsilon\}$. 

Let $K$ be a nonempty subset of $X$. For any $\varepsilon>0$, a subset $E$ of $X$ is said to be an \emph{$(n,\varepsilon)$-spanning set of $K$}, if for any $y\in K$ there exists $x\in E$ such that $d_n(x,y)\leq \varepsilon$. 
A subset $F\subseteq K$ is said to be an \emph{$(n,\varepsilon)$-separated set of $K$}, if $x,y\in F$ with $x\neq y$ implies $d_n(x,y)>\varepsilon$. 
Let $r_n(f_{1,\infty},K,\varepsilon)$ be the minimal cardinality of any $(n,\varepsilon)$-spanning set of $K$, and $s_n(f_{1,\infty},K,\varepsilon)$ be the maximal cardinality of any $(n,\varepsilon)$-separated set of $K$. 

For any $s>0$ and $\emptyset \neq K\subseteq X$, let
\begin{align*}
\overline{h}_{span}(f_{1,\infty},K,s)&=\lim_{\varepsilon\to 0}\mathop{\lim\sup}_{n\to\infty}\frac{1}{n^s}\log r_n(f_{1,\infty},K,\varepsilon);\\
\underline{h}_{span}(f_{1,\infty},K,s)&=\lim_{\varepsilon\to0}\mathop{\lim\inf}_{n\to\infty}\frac{1}{n^s}\log r_n(f_{1,\infty},K,\varepsilon);\\
\overline{h}_{sep}(f_{1,\infty},K,s)&=\lim_{\varepsilon\to0}\mathop{\lim\sup}_{n\to\infty}\frac{1}{n^s}\log s_n(f_{1,\infty},K,\varepsilon);\\
\underline{h}_{sep}(f_{1,\infty},K,s)&=\lim_{\varepsilon\to0}\mathop{\lim\inf}_{n\to\infty}\frac{1}{n^s}\log s_n(f_{1,\infty},K,\varepsilon).
\end{align*}

\begin{definition}\label{hdsep}
Let $(X,f_{1,\infty})$ be an NDS, and $K$ be a nonempty subset of $X$. For any $s>0$, we define
\begin{align*}
\overline{h}_d(f_{1,\infty},K,s)&=\overline{h}_{span}(f_{1,\infty},K,s)=\overline{h}_{sep}(f_{1,\infty},K,s),\\
\underline{h}_d(f_{1,\infty},K,s)&=\underline{h}_{span}(f_{1,\infty},K,s)=\underline{h}_{sep}(f_{1,\infty},K,s).
\end{align*}
\end{definition}

\begin{lemma}\cite{LZW2019}\label{l:hd}
Let $(X,d)$ be a compact metric space, $(X,f_{1,\infty})$ be an NDS, and $K$ be a nonempty subset of $X$. Then the two notations of topological entropy dimensions of $f_{1,\infty}$ on the set $K$ are equivalent, i.e., for any $s>0$,
\[
\overline{h}_d(f_{1,\infty},K,s)=\overline{h}_K(f_{1,\infty},s),\quad \underline{h}_d(f_{1,\infty},K,s)=\underline{h}_K(f_{1,\infty},s).
\]
\end{lemma}

\subsection{Pesin topological entropy dimension} 
Let $X$ be a compact topological space, and $(X,f_{1,\infty})$ be an NDS. For any finite open cover $\mathcal{U}$ of $X$, we denote by $S_m(\mathcal{U})$ the set of all strings $\mathbf{U}=(U_{i_0},U_{i_1},\dots,U_{i_{m-1}})$ of length $m=m(\mathbf{U})$ with $U_{i_j}\in\mathcal{U}$. And put $S(\mathcal{U})=\bigcup_{m\geq 0} S_m(\mathcal{U})$. For a given string $\mathbf{U}=(U_{i_0},U_{i_1},\dots,U_{i_{m-1}})\in S(\mathcal{U})$, let 
\[X(\mathbf{U})=X_{f_{1,\infty}}(\mathbf{U})=\left\{x\in X: f_1^j(x)\in U_{i_j},j=0,1,\dots,m(\mathbf{U})-1 \right\}.
\]
It is easy to see that $X(\mathbf{U})=\bigcap_{j=0}^{m(\mathbf{U})-1}f_1^{-j}(U_{i_j})$. We say a collection of strings $\mathcal{G}$ covers the subset $K\subseteq X$ if $K\subseteq \bigcup_{\mathbf{U}\in \mathcal{G}}X(\mathbf{U})$.

For any $\alpha\in\mathbb{R}$, $s>0$ and nonempty subset $K\subseteq X$, let 
\begin{equation}\label{M}
M(f_{1,\infty},K,s,\mathcal{U},\alpha,N)=\inf\left\{\sum\limits_{\mathbf{U}\in \mathcal{G} } e^{-\alpha m(\mathbf{U})^s}\right\},
\end{equation}
where the infimum is taken over all finite or countable collections of strings $\mathcal{G}_{N,f_{1,\infty}}\subseteq S(\mathcal{U})$ such that $m(\mathbf{U})\geq N$ for all $\mathbf{U}\in\mathcal{G}_{N,f_{1,\infty}}$ and $\mathcal{G}_{N,f_{1,\infty}}$ covers $K$.

Then taking the limit of $M(f_{1,\infty},K,s,\mathcal{U},\alpha,N)$ as $N\to\infty$,
\[m(f_{1,\infty},K,s,\mathcal{U},\alpha)=\lim_{N\to\infty} M(f_{1,\infty},K,s,\mathcal{U},\alpha,N).\]
One can check that $m(f_{1,\infty},K,s,\mathcal{U},\alpha)\leq m(f_{1,\infty},Y,s,\mathcal{U},\alpha')$ when $\alpha\geq\alpha'$. By ~\cite[Proposition 1.2]{Pes1997}, there exists a critical value $\alpha$ such that $m(f_{1,\infty},K,s,\mathcal{U},\alpha)$ jumps from $+\infty$ to $0$. So we denote this critical value by 
\begin{align*}
D(f_{1,\infty},K,s,\mathcal{U})&=\inf\{\alpha: m(f_{1,\infty},K,s,\mathcal{U},\alpha)=0 \} \\
	&=\sup\{\alpha: m(f_{1,\infty},K,s,\mathcal{U},\alpha)=\infty \}.
\end{align*}

\begin{definition}
The \emph{Pesin $s$-topological entropy of $f_{1,\infty}$ on the set $K$} is given by 
\[
D(f_{1,\infty},K,s)=\sup\{D(f_{1,\infty},K,s,\mathcal{U}):\mathcal{U} \text{ is a open cover of }X\}.
\] 
Especially, $D(f_{1,\infty},K,1)$ corresponds to the Pesin topological entropy of $f_{1,\infty}$ on the set $K$, i.e., $D(f_{1,\infty},K,1)=Ph(f_{1,\infty},K)$ \cite{li2015,JY2021}.
\end{definition}

\begin{proposition}
\begin{enumerate}
\item The map $s\to D(f_{1,\infty},K,s)$ is non-negative and non-increasing with $s;$
\item There exists $s_0>0$ such that 
\[
D(f_{1,\infty},K,s)=
\begin{cases}
+\infty,& 0<s<s_0;\\
0, &s>s_0.
\end{cases}
\]
\end{enumerate}
\end{proposition}

\begin{definition}
The \emph{Pesin topological entropy dimension} of $f_{1,\infty}$ on the set $K$ is given by the critical value of the parameter $s$, where $D(f_{1,\infty},K,s)$ jumps from $\infty$ to $0$:
\begin{align*}
D(f_{1,\infty},K)&=\sup\{s:D(f_{1,\infty},K,s)=+\infty\}\\
&=\inf\{s:D(f_{1,\infty},K,s)=0\}.
\end{align*}
\end{definition}

\begin{remark}
We can also use Bowen's method \cite{Bow1973} to define the topological entropy dimension, however, these two types of entropy dimensions are equivalent theoretically, see ~\cite{PP1984,li2015}. 
\end{remark}

\section{Topological properties of topological entropy dimensions}

\subsection{Properties of the classical entropy dimension}
Kuang et al. \cite{KCL2013} and Li et al. \cite{LZW2019} presented many results for the classical topological entropy dimension, so in this section we focus mainly on the properties that are not covered in these literature.

Given a sequence of continuous selfmaps $f_{1,\infty}=\{f_i\}_{i=1}^{\infty}$ defined on $X$, we say $f_i\in f_{1,\infty}$ is \emph{expanding}, if there exists $\lambda>1$ such that $d(f_i(x),f_i(y))>\lambda d(x,y)$ for all $x,y\in X$. The sequence $f_{1,\infty}$ is \emph{forward expansive} or \emph{positively expansive} if two points are eventually separated by more than some fixed distance, that is, there exists an \emph{expansive constant} $\delta>0$ such that for any distinct $x, y\in X$ then $\sup_{n\in \mathbb{N}} d(f_1^n(x),f_1^n(y))>\delta$. If all $f_i$ are homeomorphisms, we say a finite open cover $\mathcal{U}$ of $X$ is a \emph{generator} of $(X,f_{1,\infty})$ if for every sequence $\{A_n\}_{n\in \mathbb{Z}}$ of members of $\mathcal{U}$, the set $\bigcap_{n=-\infty}^{\infty}f_1^{-n}(\overline{A}_n)$ contains at most one point of $X$. These concepts are modified from ~\cite{Wal2000a,TD2014a}.

\begin{proposition}\label{prof:generator}
Let $(X,d)$ be a compact metric space, $(X,f_{1,\infty})$ be an NDS, and $\{\mathcal{U}_n\}$ be a sequence of finite open covers of $X$ with $\lim_{n\to\infty}|\mathcal{U}_n|=0$. Then
\[
\lim_{n\to\infty}\overline{D}_X'(f_{1,\infty},\mathcal{U}_n)=\overline{D}_X(f_{1,\infty})\ \text{ and}\ \lim_{n\to\infty}\underline{D}_X'(f_{1,\infty},\mathcal{U}_n)=\underline{D}_X(f_{1,\infty}).
\]
In particular, if $\mathcal{U}$ is a generating open cover of $X$, i.e., $\lim_{n\to\infty}|\bigvee_{j=0}^{n-1}f_1^{-j}(\mathcal{U})|\to 0$, then $\overline{D}_X'(f_{1,\infty},\mathcal{U})=\overline{D}_X(f_{1,\infty})$ and $\underline{D}_X'(f_{1,\infty},\mathcal{U})=\underline{D}_X(f_{1,\infty})$.
\end{proposition}

\begin{proof}
We prove only the first equation, the other proof is similar. For any finite open cover $\mathcal{V}$ of $X$, and $\delta$ be its Lebesgue number (A \emph{Lebesgue number} for $\mathcal{V}$ is a number $\delta>0$ such that any subset $Z\subseteq X$ having diameter less than $\delta$ is contained in some $V\in \mathcal{V}$). Then for each $n$ with $|\mathcal{U}_n|<\delta$, we have $\mathcal{V}\preceq \mathcal{U}_n$. Then $\overline{D}_X'(f_{1,\infty},\mathcal{U}_n)\geq \overline{D}_X'(f_{1,\infty},\mathcal{V})$, and  
\[
\mathop{\lim\inf}_{n\to\infty}\overline{D}_X'(f_{1,\infty},\mathcal{U}_n)\geq \overline{D}_X'(f_{1,\infty},\mathcal{V}).
\] Since $\mathcal{V}$ is arbitrary, it follows that 
\[
\mathop{\lim\inf}_{n\to\infty}\overline{D}_X'(f_{1,\infty},\mathcal{U}_n)\geq \overline{D}_X'(f_{1,\infty}).
\]
On the other hand, it is obvious that $\overline{D}_X'(f_{1,\infty},\mathcal{U}_n)\leq \overline{D}_X(f_{1,\infty})$. 
Then 
\[
\mathop{\lim\sup}_{n\to\infty}\overline{D}_X'(f_{1,\infty},\mathcal{U}_n)\leq \overline{D}_X'(f_{1,\infty}).
\]
By ~\Cref{p:equiv}, we have $\lim_{n\to\infty}\overline{D}_X'(f_{1,\infty},\mathcal{U}_n)=\overline{D}_X(f_{1,\infty})$.

If $\mathcal{U}$ is a generating open cover of $X$, then $\mathcal{U}_n:=\bigvee_{j=0}^{n-1}f_1^{-j}(\mathcal{U})$ is also an open cover of $X$ and satisfies $\lim_{n\to\infty}|\mathcal{U}_n| = 0$. Therefore $\overline{D}_X'(f_{1,\infty},\mathcal{U})=\overline{D}_X(f_{1,\infty})$ follows from the above discussion.
\end{proof}

\begin{lemma}~\cite[Theorem 16]{TD2014a}\label{lemma:forward}
Let $f_{1,\infty}$ be a sequence of homeomorphisms from $(X,d)$ to itself. Then $f_{1,\infty}$ is forward expansive if and only if this system has a generator.
\end{lemma}

\begin{theorem}\label{generator}
Let $f_{1,\infty}$ be forward expansive homeomorphisms on the compact metric space $(X,d)$. Then $h_{top}(f_{1,\infty},X)$ is finite, and there exists a generator $\mathcal{U}$ of $(X,f_{1,\infty)}$ such that 
\[
\underline{D}'_X(f_{1,\infty},\mathcal{U})=\underline{D}_X(f_{1,\infty}) \,\text{and}\ \overline{D}'_X(f_{1,\infty},\mathcal{U})=\overline{D}_X(f_{1,\infty}).
\]
\end{theorem}

\begin{proof}
This result follows from ~\cite[Theorem 7.11]{Wal2000a}, ~\Cref{prof:generator} and ~\Cref{lemma:forward}.
\end{proof}

\begin{theorem}
Let $(X,f_{1,\infty})$ be an NDS, and
$K$ be a nonempty $f_{1,\infty}$-invariant subset of $X$, i.e., $f_i^{-1}(K)=K$ for all $i\geq 1$. Then for any $1\leq i\leq j<\infty$, 
\[
\overline{D}_K(f_{i,\infty})\leq \overline{D}_K(f_{j,\infty}), \quad \underline{D}_K(f_{i,\infty})\leq \underline{D}_K(f_{j,\infty}).
\]	
\end{theorem}

\begin{proof}
We prove only the first inequality, the second inequality is similar. For this purpose, it suffices to show that for any $s>0$ and $i\geq 1$, we have $\overline{h}_K(f_{i,\infty},s)\leq \overline{h}_K(f_{i+1,\infty},s)$. 

For any open cover $\mathcal{U}$ of $X$, it is clear that 
\begin{align*}
\mathcal{U}_i^n|_K&=\left(\mathcal{U}\vee f_i^{-1}(\mathcal{U})\vee f_i^{-1}\circ f_{i+1}^{-1}(\mathcal{U})\vee\dots\vee f_i^{-1}\circ f_{i+1}^{-(n-2)}(\mathcal{U})\right)\big|_K\\
&=\Big(\mathcal{U}\vee f_i^{-1}\big(\mathcal{U}\vee f_{i+1}^{-1}(\mathcal{U})\vee\dots\vee f_{i+1}^{-(n-2)}(\mathcal{U}) \big)\Big)\Big|_K\\
&=\left(\mathcal{U}\vee f_i^{-1}(\mathcal{U}_{i+1}^{n-1}) \right)\big|_K\\
&=\mathcal{U}|_K\vee \left(f_i^{-1}(\mathcal{U}_{i+1}^{n-1})\right)\big|_K \\
&=\mathcal{U}|_K\vee f_i^{-1}(\mathcal{U}_{i+1}^{n-1}|_K) \quad (\text{since $K$ is $f_{1,\infty}$-invariant}).
\end{align*}

Therefore, 
\begin{align*}
\overline{h}_K(f_{i,\infty},s)&=\mathop{\lim\sup}_{n\to\infty}\frac{1}{n^s}\log \mathcal{N}(\mathcal{U}_i^n|_K)\\
&\leq \mathop{\lim\sup}_{n\to\infty}\frac{1}{n^s}\log \mathcal{N}(\mathcal{U}|_K)+\mathop{\lim\sup}_{n\to\infty}\frac{1}{n^s}\log \mathcal{N}\left(f_i^{-1}(\mathcal{U}_{i+1}^{n-1}|_K)\right)\\
&\leq \mathop{\lim\sup}_{n\to\infty}\frac{1}{n^s}\log \mathcal{N}(\mathcal{U})+\mathop{\lim\sup}_{n\to\infty}\frac{1}{(n-1)^s}\log \mathcal{N}(\mathcal{U}_{i+1}^{n-1}|_K)\\
&=\overline{h}_K(f_{i+1,\infty},s).
\end{align*}

Now for any $s_0>\overline{D}_K(f_{i+1,\infty},s)$, then $\overline{h}_K(f_{i+1,\infty},s_0)=0$ implies that $\overline{h}_K(f_{i,\infty},s_0)=0$. Hence $s_0\geq \overline{D}_K(f_{i,\infty})$, and this further implies that $\overline{D}_K(f_{i+1,\infty},s)\geq \overline{D}_K(f_{i,\infty})$.
\end{proof}

\begin{remark}
The condition that $K$ to be $f_{1,\infty}$-invariant is necessary, since by ~\Cref{onK}, $f_i^{-1}(\mathcal{U})|_K=(f_i|_K)^{-1}(\mathcal{U}|_K)$, and it is not true for $\mathcal{N}(f_i^{-1}(\mathcal{U})|_K)\leq \mathcal{N}(\mathcal{U}|_K)$ in general. 
\end{remark}

\begin{corollary}
Let $(X,f_{1,\infty})$ be an NDS. Then for any $1\leq i\leq j<\infty$, 
\[
\overline{D}_X(f_{i,\infty})\leq \overline{D}_X(f_{j,\infty}), \quad \underline{D}_X(f_{i,\infty})\leq \underline{D}_X(f_{j,\infty}).
\]		
\end{corollary}

\begin{theorem}\label{T:hk}
Let $(X,f_{1,\infty})$ be an NDS, and $K$ be a nonempty subset of $X$. Then for any $k\in\mathbb{N}$, 
\[
\overline{D}_K(f_{1,\infty}^k)\leq \overline{D}_K(f_{1,\infty}),\quad \underline{D}_K(f_{1,\infty}^k)\leq \underline{D}_K(f_{1,\infty}).
\]
\end{theorem}

\begin{proof}
To prove this theorem, we need only to show that, for any $s>0$ and $k\in \mathbb{N}$, then 
\[
\overline{h}_K(f_{1,\infty}^k,s) \leq k^s\cdot \overline{h}_K(f_{1,\infty},s),\quad \underline{h}_K(f_{1,\infty}^k,s)\leq k^s\cdot \underline{h}_K(f_{1,\infty},s).
\]

For this purpose, let $g_j=f_{(j-1)k+1}^k$ for $j\geq 1$. Then $f_{1,\infty}^k=\{f_{(j-1)k+1}^k\}_{j=1}^{\infty}=\{g_j\}_{j=1}^{\infty}$,
\begin{equation}\label{g1j}
f_1^{jk}=f_{(j-1)k+1}^{k}\circ \dots\circ f_{k+1}^k\circ f_1^k=g_j\circ\dots\circ g_1=g_1^j.
\end{equation}
From this, it is easy to verify that if $E$ is an $(kn,\varepsilon)$-spanning set of $K$ with respect to $f_{1,\infty}$, then $E$ is also an $(n,\varepsilon)$-spanning set of $K$ with respect to $f_{1,\infty}^k$. Hence $r_n(f_{1,\infty}^k,K,\varepsilon)\leq r_{nk}(f_{1,\infty},K,\varepsilon)$. By ~\Cref{hdsep} and ~\Cref{l:hd},
\begin{align*}
\overline{h}_K(f_{1,\infty}^k,s)&= \lim_{\varepsilon\to 0}\mathop{\lim\sup}_{n\to\infty}\frac{1}{n^s}\log r_n(f_{1,\infty}^k,K,\varepsilon)\\
&\leq \lim_{\varepsilon\to 0}\mathop{\lim\sup}_{n\to\infty} \frac{k^s}{(nk)^s}\log r_{nk}(f_{1,\infty},K,\varepsilon)\\
&=k^s \cdot\overline{h}_K(f_{1,\infty},s).
\end{align*}

Similarly, we could prove $\underline{h}_K(f_{1,\infty}^k,s) \leq k^s\cdot \underline{h}_K(f_{1,\infty},s)$.

The rest proof is clear since if $s_0>\overline{D}_K(f_{1,\infty})$, then $\overline{h}_K(f_{1,\infty},s_0)=0$ implies $\overline{h}_K(f_{1,\infty}^k,s_0)=0$, and it further implies that $s_0>\overline{D}_K(f_{1,\infty}^k)$. Thus $\overline{D}_K(f_{1,\infty})\geq \overline{D}_K(f_{1,\infty}^k)$ by the arbitrariness of $s_0$. Similarly, we have $\underline{D}_K(f_{1,\infty}^k)\leq \underline{D}_K(f_{1,\infty})$. 
\end{proof}

\begin{remark}
Kuang et al. ~\cite[Proposition 3.2]{KCL2013} proved the upper topological entropy dimension inequality when $K=X$. And Li et al. ~\cite[Theorem 3.7]{LZW2019} proved a similar result when the sequence $f_{1,\infty}$ is equicontinuous (see ~\cref{D:equi2}). Particularly when $f_i=f$ for all $i\geq 1$ then we further have $\underline{D}_K(f_{1,\infty}^k)= \underline{D}_K(f_{1,\infty})$, see ~\cite[Theorem 3.2]{KCML2014}.
\end{remark}

\begin{theorem}
Let $(X,f_{1,\infty})$ be an NDS, and $K$ be a nonempty subset of $X$. If the sequence $f_{1,\infty}$ has period $k$, then 
\[\overline{D}_K(f_{1,\infty}^k)= \overline{D}_K(f_{1,\infty})\geq \underline{D}_K(f_{1,\infty}^k)= \underline{D}_K(f_{1,\infty}).
\]
\end{theorem}

\begin{proof}
To prove this theorem, by ~\Cref{T:hk}, we need only to prove that for any $s>0$ and some open cover $\mathcal{U}$ of $X$, 
\[\overline{h}_K(f_{1,\infty}^k,s,\mathcal{U})\geq k^s\cdot\overline{h}_K(f_{1,\infty},s,\mathcal{U})\quad \text{and}\quad \underline{h}_K(f_{1,\infty}^k,s,\mathcal{U})\geq k^s\cdot \underline{h}_K(f_{1,\infty},s,\mathcal{U}).
\]

Note that $f_{1,\infty}=\{f_1,\dots,f_k,f_1,\dots,f_k,\dots\}$ and $f_{1,\infty}^k=\{f_1^k,f_1^k,\dots\}$ . For any $t\in\mathbb{N}$ and $1\leq i\leq k-1$, by ~\eqref{g1j}, then $f_1^{kt}=(f_1^k)^t$ and
\begin{align*}
f_1^{tk+i}&=(f_{tk+i}\circ f_{tk+i-1}\circ \dots \circ f_{tk+1})\circ\left( f_{tk}\circ \dots\circ f_{k+1}\circ (f_k\circ \dots\circ f_1)\right)\\
&=(f_i\circ \dots f_1)\circ f_1^{tk}\\
&=f_1^i\circ f_1^{kt}.
\end{align*}
Thus 
\begin{align}\label{reverse}
f_1^{-(tk+i)}=(f_1^{tk+i})^{-1}=f_1^{-kt}\circ f_1^{-i}=(f_1^{kt})^{-1}\circ f_1^{-i}.
\end{align}

We begin by proving the first inequality. For any $n\in\mathbb{N}$, let $n=t_0k+r$, where $t_0=[n/k]$ is the integer part of $n/k$. Then
\begin{align*}
\bigvee_{j=0}^{n-1}f_1^{-j}(\mathcal{U})&=\big[\mathcal{U}\vee f_1^{-1}(\mathcal{U})\vee\dots\cap f_1^{-(k-1)}(\mathcal{U})\big] \vee\big[ f_1^{-k}(\mathcal{U})\vee\dots\vee f_1^{-(2k-1)}(\mathcal{U})\big] \nonumber \\
	&\qquad \vee\dots \vee \big[f_1^{-(t_0-1)k}(\mathcal{U})\vee f_1^{-((t_0-1)k+1)}(\mathcal{U})\dots\vee f_1^{-(t_0k-1)}(\mathcal{U})\big] \nonumber \\ 
	&\qquad \vee\big[f_1^{-t_0k}(\mathcal{U})\vee f_1^{-(t_0k+1)}(\mathcal{U})\vee\dots\vee f_1^{-(t_0k+r-1)}(\mathcal{U})\big]\nonumber\\ 
	&=\big[\mathcal{U}\vee f_1^{-1}(\mathcal{U})\vee\dots\cap f_1^{-(k-1)}(\mathcal{U})\big] \vee f_1^{-k}\big(\mathcal{U}\vee f_1^{-1}(\mathcal{U})\vee\dots\vee f_1^{-(k-1)}(\mathcal{U})\big)\nonumber \\ 
	&\qquad \vee \dots \vee f_1^{-(t_0-1)k}\big(\mathcal{U}\vee f_1^{-1}(\mathcal{U})\vee\dots\vee f_1^{-(k-1)}(\mathcal{U}) \big)\nonumber \\
	&\qquad \vee f_1^{-t_0k}\big(\mathcal{U}\vee f_1^{-1}(\mathcal{U})\vee\dots\vee f_1^{-(r-1)}(\mathcal{U}) \big)\quad (\text{By} ~\eqref{reverse})\nonumber \\
	&=\big(\mathcal{U}_1^k\vee (f_1^k)^{-1}(\mathcal{U}_1^k)\vee\dots\vee (f_1^{(t_0-1)k})^{-1}(\mathcal{U}_1^k)\big)\vee (f_1^{t_0k})^{-1}(\mathcal{U}_1^r)\nonumber \\
	&=\left(\bigvee_{j=0}^{t_0-1}(f_1^k)^{-j}(\mathcal{U}_1^k)\right) \vee (f_1^{t_0k})^{-1}(\mathcal{U}_1^r).
\end{align*}
Then
\begin{align*}
\mathcal{N}(\mathcal{U}_1^n|_K)
&=\mathcal{N}\left(\bigvee_{j=0}^{n-1}f_1^{-j}(\mathcal{U})\Big|_K\right)\\
	&\leq \mathcal{N}\left(\big(\bigvee_{j=0}^{t_0-1}(f_1^k)^{-j}(\mathcal{U}_1^k)\big)\Big|_K\right)\cdot \mathcal{N}\left((f_1^{-t_0k}(\mathcal{U}_1^r))|_K\right)\\
	&\leq \mathcal{N}\left(\big(\bigvee_{j=0}^{t_0-1}(f_1^k)^{-j}(\mathcal{U}_1^k)\big)\Big|_K\right)\cdot \mathcal{N}(f_1^{-t_0k}(\mathcal{U}_1^r)) \quad \text{(By ~\Cref{onK})}\\
	&\leq \mathcal{N}\left(\big(\bigvee_{j=0}^{t_0-1}(f_1^k)^{-j}(\mathcal{U}_1^k)\big)\Big|_K\right)\cdot \mathcal{N}(\mathcal{U}_1^r) \quad \text{(By ~\Cref{L:useful})}.
\end{align*}
This follows that
\begin{equation}\label{logn}
\log \mathcal{N}\left(\big(\bigvee_{j=0}^{t_0-1}(f_1^k)^{-j}(\mathcal{U}_1^k)\big)\big|_K\right)\geq \log \mathcal{N}(\mathcal{U}_1^n|_K) -\log \mathcal{N}(\mathcal{U}_1^r).
\end{equation}
Then
\begin{align*}
\overline{h}_K(f_{1,\infty}^k,s,\mathcal{U}_1^k)&=\mathop{\lim\sup}_{t_0\to\infty}\frac{1}{t_0^s}\log\mathcal{N}\left(\big(\bigvee_{j=0}^{t_0-1}(f_1^k)^{-j}(\mathcal{U}_1^k)\big)\Big|_K \right)\\
	&\geq \mathop{\lim\sup}_{t_0\to\infty}\frac{1}{t_0^s}\left[ \log \mathcal{N}(\mathcal{U}_1^{n}|_K)-\log \mathcal{N}(\mathcal{U}_1^r) \right] \quad (\text{By}~\eqref{logn})\\
	&= \mathop{\lim\sup}_{t_0\to\infty}\frac{1}{t_0^s}\log \mathcal{N}(\mathcal{U}_1^n|_K)\\
	&=\mathop{\lim\sup}_{t_0\to\infty}\frac{1}{n^s}\frac{(t_0k+r)^s}{t_0^s}\log \mathcal{N}(\mathcal{U}_1^n|_K)\\
	&= k^s\cdot\mathop{\lim\sup}_{n\to\infty}\frac{1}{n^s}\log \mathcal{N}(\mathcal{U}_1^n|_K)\\
	&=k^s\cdot\overline{h}_K(f_{1,\infty},s,\mathcal{U}).
\end{align*}
Therefore, $\overline{h}_K(f_{1,\infty}^k,s)\geq \overline{h}_K(f_{1,\infty}^k,s,\mathcal{U}_1^k)\geq k^s\cdot\overline{h}_K(f_{1,\infty},s,\mathcal{U})$. By the arbitrariness of $\mathcal{U}$, it follows that $\overline{h}_K(f_{1,\infty}^k,s)\geq k^s\cdot\overline{h}_K(f_{1,\infty},s)$. Similarly we could prove $\underline{h}_K(f_{1,\infty}^k,s,\mathcal{U})\geq k^s\cdot \underline{h}_K(f_{1,\infty},s,\mathcal{U})$. 

Finally, for any $s_0>\overline{D}_K(f_{1,\infty}^k)$, then $\overline{h}_K(f_{1,\infty}^k,s_0)=0$ implies $\overline{h}_K(f_{1,\infty},s_0)=0$. Thus $s_0\geq \overline{D}_K(f_{1,\infty})$ and it follows that $\overline{D}_K(f_{1,\infty}^k)\geq \overline{D}_K(f_{1,\infty})$. Then by \Cref{T:hk}, we have $\overline{D}_K(f_{1,\infty}^k)= \overline{D}_K(f_{1,\infty})$, and similarly $\underline{D}_K(f_{1,\infty}^k)= \underline{D}_K(f_{1,\infty})$.
\end{proof}

\begin{definition}\label{D:equi2}
Let $(X,d)$ be a compact metric space. We say the sequence of selfmaps $f_{1,\infty}=\{f_i\}_{i=1}^{\infty}$ is \emph{equicontinuous}, if for any $\varepsilon>0$, there exists $\delta=\delta(\varepsilon)>0$ such that $d(f_i(x),f_i(y))<\varepsilon$ for all $i\geq 1$ whenever $d(x,y)<\delta$.
\end{definition}

\begin{lemma} ~\cite[Proposition 3.5]{KCL2013} \label{KCLP}
Let $(X,d)$ be a compact metric space, and $f_{1,\infty}$ be a sequence of equicontinuous selfmaps of $X$. Then for any $k\geq 1$, 
\[
\overline{D}_X(f_{1,\infty}^k)= \overline{D}_X(f_{1,\infty}).
\]
\end{lemma}

\begin{theorem}
Let $(X,d)$ be a compact metric space, $(X,f_{1,\infty})$ be an NDS, and $K$ be a nonempty subset of $X$. If the sequence $f_{1,\infty}$ is equicontinuous, then for any $k\geq 1$,
\[\overline{D}_K(f_{1,\infty}^k)= \overline{D}_K(f_{1,\infty})\geq \underline{D}_K(f_{1,\infty}^k)= \underline{D}_K(f_{1,\infty}).
\]{}
\end{theorem}

\begin{proof}
This proof is essentially similar to the proof of ~\Cref{KCLP}, so we will just sketch the general idea. Since the sequence $f_{1,\infty}$ is equicontinuous, then for any $\varepsilon>0$, we can construct an $\delta(\varepsilon)>0$ such that $\delta(\varepsilon)\to 0$ as $\varepsilon\to 0$ and $d(f_i^j(x),f_i^j(y))\leq \delta(\varepsilon)$ whenever $i\geq 1$, $j=1,2,\dots,k-1$ and $d(x,y)\leq \varepsilon$ (see ~\cite{KS1996} for details). Therefore, if $F$ is an $(nk,\delta(\varepsilon))$-separated set for $K$ with respect to $f_{1,\infty}$, then $F$ is also an $(n,\varepsilon)$-separated set for $K$ with respect to $f_{1,\infty}^k$. Thus $s_n(f_{1,\infty}^k,K,\varepsilon)\geq s_{nk}(f_{1,\infty},K,\delta(\varepsilon))$. This rest of proof follows directly from ~\Cref{hdsep} and ~\Cref{l:hd}, so we omit it.
\end{proof}

\subsection{Properties of the Pesin entropy dimension}

\begin{proposition}\label{p:properties}
Let $(X,f_{1,\infty})$ be an NDS, and $\mathcal{U},\mathcal{U}_i,\mathcal{V}$ be open covers of $X$. Then for any $\alpha\in \mathbb{R}$, $s>0$ and nonempty subsets $K,K_i\subseteq X$. 
\begin{enumerate}
\item If $\mathcal{U}\preceq \mathcal{V}$, then 
\begin{align*}
m(f_{1,\infty},K,s,\mathcal{U},\alpha)&\leq m(f_{1,\infty},K,s,\mathcal{V},\alpha),\\
D(f_{1,\infty},K,s,\mathcal{U})&\leq D(f_{1,\infty},K,s,\mathcal{V});
\end{align*}

\item $\max_{i=1}^m D(f_{1,\infty},K,s,\mathcal{U}_i)\leq D(f_{1,\infty},K,s,\bigvee_{i=1}^m\mathcal{U}_i);$

\item if $K\subseteq Z\subseteq X$, then $D(f_{1,\infty},K,s)\leq D(f_{1,\infty},Z,s)$ and $D(f_{1,\infty},K)\leq D(f_{1,\infty},Z);$

\item 
\begin{align*}
m(f_{1,\infty},\bigcup_{i\geq 1}K_i,s,\mathcal{U},\alpha)&\leq \sum_{i\geq 1} m(f_{1,\infty},K_i,s,\mathcal{U},\alpha),\\
D(f_{1,\infty},\bigcup_{i\geq 1}K_i,s)&= \sup_{i\geq 1}D(f_{1,\infty},K_i,s),\\
D(f_{1,\infty},\bigcup_{i\geq 1}K_i)&= \sup_{i\geq 1}D(f_{1,\infty},K_i).
\end{align*}
\end{enumerate}
\end{proposition}

\begin{proof}
(1) Let $\mathbf{V}=(V_{i_0},V_{i_1},\dots,V_{i_{m-1}})\in S_m(\mathcal{V})$. If $\mathcal{U}\preceq \mathcal{V}$, then for any $V_{i_j}\in\mathcal{V}$, there exists $U(V_{i_j})\in\mathcal{U}$ such that $V_{i_j}\subseteq U(V_{i_j})$. We associate each $\mathbf{V}$ with $\mathbf{U}(\mathbf{V})=(U(V_{i_0}),U(V_{i_1}),\dots,U(V_{i_{m-1}}))$. If the collection of strings $\mathcal{G}_{N,f_{1,\infty}}$ covers $K$, then the collection of strings $\{\mathbf{U}(\mathbf{V}):\mathbf{V}\in\mathcal{G}_{N,f_{1,\infty}}\}$ also covers $K$. Thus $M(f_{1,\infty},K,s,\mathcal{U},\alpha,N)\leq M(f_{1,\infty},K,s,\mathcal{V},\alpha,N)$ and it follows that $m(f_{1,\infty},K,s,\mathcal{U},\alpha)\leq m(f_{1,\infty},K,s,\mathcal{V},\alpha)$. 
Note that for any $\alpha> D(f_{1,\infty},K,s,\mathcal{V})$, $m(f_{1,\infty},K,s,\mathcal{V},\alpha)$=0 and this implies that $m(f_{1,\infty},K,s,\mathcal{U},\alpha)=0$, thus $\alpha\geq D(f_{1,\infty},K,s,\mathcal{U})$. The we have $D(f_{1,\infty},K,s,\mathcal{V})\geq D(f_{1,\infty},K,s,\mathcal{U})$.

(2) Since $\mathcal{U}_i\preceq \bigvee_{i=1}^m \mathcal{U}_i$ for any $1\leq i\leq m$, then this result follows from (1).

(3) This inequality is clear if $K\subseteq Z\subseteq X$ and the collection $\mathcal{G}_{N,f_{1,\infty}}$ covers $Z$, then $\mathcal{G}_{N,f_{1,\infty}}$ also covers $K$. 

(4) The proof of the subadditivity of $m(f_{1,\infty},K_i,s,\mathcal{U},\alpha)$ is similar to \cite[Proposition 1.1]{Pes1997}, so we omit it. By (3), then $D(f_{1,\infty},\bigcup_{i\geq 1}K_i,s)\geq \sup_{i\geq 1}D(f_{1,\infty},K_i,s)$. 
To prove the inverse inequality, given open cover $\mathcal{U}$ of $X$, for any $\alpha>D(f_{1,\infty},K_i,s,\mathcal{U})$ for all $i \geq 1$, then $m(f_{1,\infty},K_i,s,\mathcal{U},\alpha)=0$ and hence by the subadditivity of $m(f_{1,\infty},K_i,s,\mathcal{U},\alpha)$, we get $m(f_{1,\infty},\bigcup_{i\geq 1}K_i,s,\mathcal{U},\alpha)=0$. Therefore $\alpha\geq D(f_{1,\infty},\bigcup_{i\geq 1}K_i,s,\mathcal{U})$. This implies that $D(f_{1,\infty},K_i,s,\mathcal{U}) \geq D(f_{1,\infty},\bigcup_{i\geq 1}K_i,s,\mathcal{U})$ and then $D(f_{1,\infty},\bigcup_{i\geq 1}K_i,s)\leq \sup_{i\geq 1}D(f_{1,\infty},K_i,s)$ by taking the supremum of all covers $\mathcal{U}$ of $X$. The last equality $D(f_{1,\infty},\bigcup_{i\geq 1}K_i)= \sup_{i\geq 1}D(f_{1,\infty},K_i)$ follows directly from the preceding discussion.
\end{proof}

\begin{proposition}\label{P:distance}
Let $(X,d)$ be a compact metric space, $(X,f_{1,\infty})$ be an NDS, and $K$ be a nonempty subset of $X$. Then for any $s>0$, 
\[
D(f_{1,\infty},K,s)=\lim_{|\mathcal{U}|\to 0} D(f_{1,\infty},K,s,\mathcal{U}).\]
\end{proposition}

\begin{proof}
Note that $D(f_{1,\infty},K,s)$ is the supremum of $D(f_{1,\infty},K,s,\mathcal{U})$, then 
\[
D(f_{1,\infty},K,s)\geq \mathop{\lim\sup}_{|\mathcal{U}|\to0}D(f_{1,\infty},K,s,\mathcal{U})\geq \mathop{\lim\inf}_{|\mathcal{U}|\to0}D(f_{1,\infty},K,s,\mathcal{U}).
\]
To show the inverse inequality. Let $\mathcal{V}$ be a finite open cover of $X$, and $\mathcal{U}$ be another finite open cover of $X$ with diameter smaller than the Lebesgue number of $\mathcal{V}$. Immediately, we have $\mathcal{V}\preceq \mathcal{U}$. Then by (1) of Proposition ~\ref{p:properties}, it follows that $D(f_{1,\infty},K,s,\mathcal{V})\leq D(f_{1,\infty},K,s,\mathcal{U})$. Since $X$ is compact and has finite open covers of arbitrarily small diameter, then
\[
D(f_{1,\infty},K,s,\mathcal{V})\leq \mathop{\lim\inf}_{|\mathcal{U}|\to0}D(f_{1,\infty},K,s,\mathcal{U}).
\]
This implies that 
\[
D(f_{1,\infty},K,s)\leq \mathop{\lim\inf}_{|\mathcal{U}|\to0}D(f_{1,\infty},K,s,\mathcal{U}).
\]
Therefore we get $D(f_{1,\infty},K,s)=\lim_{|\mathcal{U}|\to 0} D(f_{1,\infty},K,s,\mathcal{U})$.
\end{proof}

\begin{theorem}\label{T:inequality} Let $(X,f_{1,\infty})$ be an NDS, and $K$ be a nonempty $f_{1,\infty}$-invariant subset of $X$. Then for any $1\leq i\leq j<\infty$,
\[
D(f_{i,\infty},K)\leq D(f_{j,\infty},K).
\]
\end{theorem}

\begin{proof}
We need only to prove $D(f_{k+1,\infty},K,s)\geq D(f_{k,\infty},K,s)$ holds for any $k\in\mathbb{N}$ and $s>0$. For this purpose, we have to show that for $\alpha>0$, $N>1$ and some open cover $\mathcal{U}$ of $X$, then 
\[
M(f_{k+1,\infty},K,s,\mathcal{U},\alpha,N-1)\geq M(f_{k,\infty},K,s,\mathcal{U},\alpha,N)
\]

Let $\mathcal{V}=\{V_1,\dots,V_m\}$ be an open cover of $X$. Then $X\subseteq \bigcup_{1\leq j\leq m}V_j$. For any $\mathbf{V}=(V_{i_0},\dots,V_{i_{m(\mathbf{V})-1}})\in \mathcal{G}_{N-1,f_{k+1,\infty}}$ with $m(\mathbf{V})\geq N-1$ and $K\subseteq \bigcup_{\mathbf{V}\in \mathcal{G}_{N-1,f_{k+1,\infty}}}X_{f_{k+1,\infty}}(\mathbf{V})$, where 
\begin{align*}
X_{f_{k+1,\infty}}(\mathbf{V})&=\{x\in X:f_{k+1}^j(x)\in V_{i_j},j=0,1,\dots,m(\mathbf{V})-1 \}\\
	&=V_{i_0}\cap f_{k+1}^{-1}(V_{i_1})\cap\dots\cap f_{k+1}^{-(m(\mathbf{V})-1)}(V_{i_{m(\mathbf{V})-1}}).
\end{align*}
We set $\mathbf{U}:=(V_t,\mathbf{V})=(V_t,V_{i_0},\dots,V_{i_{m(\mathbf{V})-1}})$ with some $V_t\in\mathcal{V}$ and $m(\mathbf{U})\geq N$.
It is clear that
\begin{align*}
X_{f_{k,\infty}}(\mathbf{U})&=V_t\cap f_k^{-1}(V_{i_0})\cap\dots\cap f_k^{-m(\mathbf{V})}(V_{i_{m(\mathbf{V})-1}})\\
	&=V_{t}\cap f_k^{-1}\left(V_{i_0}\cap f_{k+1}^{-1}(V_{i_1})\cap\dots\cap f_{k+1}^{-(m(\mathbf{V})-1)}(V_{i_{m(\mathbf{V})-1}}) \right)\\
	&=V_t\cap f_k^{-1}\left(X_{f_{k+1,\infty}}(\mathbf{V})\right).
\end{align*}

When $K$ is $f_{1,\infty}$-invariant, by a method proposed by \cite{JY2021}, then 
\begin{align*}
K&=f_k^{-1}(K)\subseteq \bigcup_{\mathbf{V}\in \mathcal{G}_{N-1,f_{k,\infty}}}f_k^{-1}(X_{f_{k+1,\infty}}(\mathbf{V}))\\
	&= \bigcup_{\mathbf{V}\in \mathcal{G}_{N-1,f_{k,\infty}}}f_k^{-1}\left(V_{i_0}\cap f_{k+1}^{-1}(V_{i_1})\cap\dots\cap f_{k+1}^{-(m(\mathbf{V})-1)}(V_{i_{m(\mathbf{V})-1}})\right)\bigcap X\\
	& \subseteq \bigcup_{\mathbf{V}\in \mathcal{G}_{N-1,f_{k,\infty}}}f_k^{-1}\left(V_{i_0}\cap f_{k+1}^{-1}(V_{i_1})\cap\dots\cap f_{k+1}^{-(m(\mathbf{V})-1)}(V_{i_{m(\mathbf{V})-1}})\right)\bigcap \Big(\bigcup_{1\leq i\leq m} V_i\Big)\\
	&=\bigcup_{\mathbf{V}\in \mathcal{G}_{N-1,f_{k,\infty}}} \bigcup_{1\leq t\leq m} \left(V_t\cap f_k^{-1}(V_{i_0})\cap f_{k}^{-2}(V_{i_1})\cap\dots\cap f_{k}^{-m(\mathbf{V})}(V_{i_{m(\mathbf{V})-1}})\right).
\end{align*}

Then the set $\{V_t\cap f_k^{-1}(V_{i_0})\cap f_{k}^{-2}(V_{i_1})\cap\dots\cap f_{k}^{-m(\mathbf{V})}(V_{i_{m(\mathbf{V})-1}})\}_{1\leq t\leq m, \mathbf{V}\in \mathcal{G}_{N-1,f_{k,\infty}}}=:\mathcal{G}_{N,f_{k,\infty}}$ also covers $K$.

By the definition of $M(f_{k+1,\infty},K,s,\mathcal{U},\alpha,N-1)$, note that the infimum is taken over all finite or countable collections of strings $\mathcal{G}_{N-1,f_{k-1,\infty}}$ such that $m(\mathbf{V})\geq N-1$ for all $\mathbf{V}\in \mathcal{G}_{N-1,f_{k+1,\infty}}$ and $\mathcal{G}_{N-1,f_{k+1,\infty}}$ covers $K$. Then
\begin{align*}
M(f_{k+1,\infty},K,s,\mathcal{U},\alpha,N-1)&=\inf_{\mathcal{G}_{N-1,f_{k+1,\infty}}}\left\{\sum_{\mathbf{V}\in \mathcal{G}_{N-1,f_{k+1,\infty}}}e^{-\alpha m(\mathbf{V})^s} \right\}\\
	&\geq \inf_{\mathcal{G}_{N,f_{k,\infty}}}\left\{\sum_{\mathbf{U}\in\mathcal{G}_{N,f_{k,\infty}}}e^{-\alpha m(\mathbf{U})^s} \right\}.
\end{align*}
One can easily verify that $M(f_{k+1,\infty},K,s,\mathcal{U},\alpha,N-1)\geq M(f_{k,\infty},K,s,\mathcal{U},\alpha,N)$. Now it is routine to follow the definition of Pesin $s$-topological entropy, then for all $k\in\mathbb{N}$, we have $D(f_{k+1,\infty},K,s)\geq D(f_{k,\infty},K,s)$.

For the rest of this proof, for any $s_0>D(f_{k+1,\infty},K)$, $D(f_{k+1,\infty},K,s_0)=0$ implies that $D(f_{k,\infty},K,s_0)=0$. This further implies that $s_0\geq D(f_{k,\infty},K)$ and thus we get $D(f_{k+1,\infty},K)\geq D(f_{k,\infty},K)$.
\end{proof}

\begin{theorem}\label{T:power}
Let $(X,f_{1,\infty})$ be an NDS, and $K$ be a nonempty subset of $X$. Then for any $k\in\mathbb{Z}$, 
\[
D(f_{1,\infty}^k,K)\leq D(f_{1,\infty},K).
\]
\end{theorem}

\begin{proof}
To prove this result, we need to prove that for any $k\in \mathbb{N}$ and $s>0$, then 
\[
D(f_{1,\infty}^k,K,s)\leq k^s D(f_{1,\infty},K,s).
\]
Let $\mathcal{U}$ be a finite open cover of $X$. We use $S_{n,f_{1,\infty}}(\mathcal{U})$ and $S_{n,f_{1,\infty}^k}(\mathcal{U})$ to denote the set of strings with length $n$ in $(X,f_{1,\infty})$ and $(X,f_{1,\infty}^k)$, respectively. Let $\mathcal{G}_{f_{1,\infty}}\subseteq \bigcup_{j\geq kN} S_{j,f_{1,\infty}}(\mathcal{U})$ cover $K$ and $\mathbf{U}=(U_{i_0},U_{i_1},\dots,U_{i_{j-1}})\in\mathcal{G}_{f_{1,\infty}}$. 
Notice that $f_{1,\infty}^k=\{f_1^k,f_{k+1}^k,f_{2k+1}^k,\dots\}$ and
\begin{align}\label{Xu}
X_{f_{1,\infty}}(\mathbf{U})&=\{x\in X: f_1^j(x)\in U_{i_j}, 0\leq j\leq m(\mathbf{U})-1\},\nonumber \\
	X_{f_{1,\infty}^k}(\mathbf{U})&=\{x\in X: f_1^{jk}(x)\in U_{i_j}, 0\leq j\leq m(\mathbf{U})-1\}.
\end{align}

We set $V_{i_0}=U_{i_0}, V_{i_1}=U_{i_k},\dots,V_{i_n}=U_{i_{kn}}$, $\mathbf{V}=(V_{i_0},V_{i_1},\dots,V_{i_n})$, and denote the collection of all such $\mathbf{V}$ by $\mathcal{G}_{f_{1,\infty}^k}$. 

It is clear that $X_{f_{1,\infty}}(\mathbf{U})\subseteq X_{f_{1,\infty}^k}(\mathbf{V})$, $\mathcal{G}_{f_{1,\infty}^k}$ covers $K$ and $\mathcal{G}_{f_{1,\infty}^k}\subseteq \bigcup_{j\geq N}S_{j,f_{1,\infty}^k}(\mathcal{U})$. Note that $m(\mathbf{U})-1$ can be represented as
\[
m(\mathbf{U})-1=kp+q, \quad 0\leq q<k, p\geq N.
\]
Then
\[
\frac{m(\mathbf{U})}{k}=p+\frac{q+1}{k}\leq n+1=m(\mathbf{V}).
\]
Therefore
\begin{align*}
M(f_{1,\infty},K,s,\mathcal{U},\frac{\alpha}{k^s},kN)&=\inf_{\mathcal{G}_{f_{1,\infty}} \text{covers } K}\left \{\sum_{\mathbf{U}\in\mathcal{G}_{f_{1,\infty}}} e^{-\frac{\alpha}{k^s}\cdot m(\mathbf{U})^s} \right\}\\
	&\geq \inf_{\mathcal{G}_{f_{1,\infty}^k} \text{covers } K}\left\{ \sum_{\mathbf{V}\in\mathcal{G}_{f_{1,\infty}^k}}e^{-\alpha\cdot m(\mathbf{V})^s} \right\}\\
	&=M(f_{1,\infty}^k,K,s,\mathcal{U},\alpha,N).
\end{align*}
Letting $N\to\infty$,
\[
m(f_{1,\infty},K,s,\mathcal{U},\frac{\alpha}{k^s})\geq m(f_{1,\infty}^k,K,s,\mathcal{U},\alpha).
\]
It follows that 
\[
k^s\cdot D(f_{1,\infty},K,s,\mathcal{U})\geq D(f_{1,\infty}^k,K,s,\mathcal{U}).
\]
Then taking the supremum of all open covers $\mathcal{U}$ of $X$, 
\[
D(f_{1,\infty}^k,K,s)\leq k^s D(f_{1,\infty},K,s).
\]
Now for any $s_0>D(f_{1,\infty},K)$, then $D(f_{1,\infty},K,s_0)=0$ implies $D(f_{1,\infty}^k,K,s_0)=0$. This further implies that $s_0\geq D(f_{1,\infty}^k,K)$. Therefore we have $D(f_{1,\infty}^k,K)\leq D(f_{1,\infty},K)$.
\end{proof}



\begin{theorem}\label{T:powerrule}
Let $(X,d)$ be a compact metric space, $(X,f_{1,\infty})$ be an NDS, and $K$ be a nonempty subset of $X$. If the sequence $f_{1,\infty}$ is equicontinuous. Then for any $k\in\mathbb{N}$,
\[
D(f_{1,\infty}^k,K)= D(f_{1,\infty},K).
\]
\end{theorem}
	
\begin{proof}
By ~\Cref{T:power}, it suffices to prove $D(f_{1,\infty}^k,K)\geq D(f_{1,\infty},K)$. To be more detailed, we have to prove that for any $s>0$ and $k\in \mathbb{N}$, $D(f_{1,\infty}^k,K,s)\geq k^s D(f_{1,\infty},K,s)$.

Note that $X$ is compact, and $\{f_i\}_{i= 1}^{\infty}$ is equicontinuous (See ~\Cref{D:equi2}), then for any $\varepsilon>0$, let
\[
\delta(\varepsilon)=\varepsilon+\sup_{i\geq1}\max_{1\leq j\leq k-1}\sup_{x,y\in X}\{d(f_i^j(x),f_i^j(y)):d(x,y)\leq \varepsilon\}.
\]
It is clear that $\varepsilon\to0$ implies $\delta(\varepsilon)\to 0$; And for any $U_{\lambda}\subseteq X$ with $|U_{\lambda}|\leq\varepsilon$, then $|f_i^j(U_{\lambda})|\leq \delta(\varepsilon)$ for any $i\geq 1$ and $j=1,2,\dots, k-1$. 

Based on this, let $\mathcal{U}$ be a finite open cover of $X$. We will construct a new open cover $\mathcal{U}'$ of $X$ from $\mathcal{U}$. For any $U_{\lambda}\in\mathcal{U}$, there exists a sequence of open sets $(U_{\lambda})_{j}\subseteq X$ such that for any $i\geq 1$, 
\begin{align}\label{equi}
f_i^j(U_{\lambda})\subseteq (U_{\lambda})_{j} \text{ and } |U_{\lambda}|\to0 \text{ implies } |(U_{\lambda})_{j}|\to 0,\quad 1\leq j\leq k-1.
\end{align}
Let 
\[
\mathcal{U}'=\{V:V\in\mathcal{U}\text{ or }V=(U_{\lambda})_{j}, 1\leq j\leq k-1\}.
\]
We can check that $\mathcal{U}'$ is indeed an open cover of $X$. For any $\mathcal{G}_{f_{1,\infty}^k} \subseteq \bigcup_{n\geq N}S_{n,f_{1,\infty}^k}(\mathcal{U})$ covers $K$, and $\mathbf{U}=(U_0,U_1,\dots,U_{n-1})\in \mathcal{G}_{f_{1,\infty}^k}$ with length $m(\mathbf{U})=n$, where $U_{\lambda}\in \mathcal{U}$ for $0\leq \lambda\leq n-1$, we set 
\[
\mathbf{V}=\left(U_0,(U_0)_1,\dots,(U_0)_{k-1},\dots,U_{n-1},(U_{n-1})_1,\dots, (U_{n-1})_{k-1}\right)\in S_{n,f_{1,\infty}}(\mathcal{U}').
\]
It is clear that $m(\mathbf{V})=kn$. Remember that
\[X_{f_{1,\infty}^k}(\mathbf{U})=\{x\in X: f_1^{jk}(x)\in U_{j}, 0\leq j\leq n-1\}.
\]
Then for any $x\in X_{f_{1,\infty}^k}(\mathbf{U})$, $f_1^{j k} (x)\in U_{j}$ for $0\leq j \leq n-1$. By ~\eqref{equi}, we have
\begin{align*}
x&\in U_0,f_1(x)\in (U_0)_1,\dots,f_1^{k-1}(x)\in (U_0)_{k-1},\\
	f_1^k(x)&\in U_1,\dots, f_1^{2k-1}(x)=f_{k+1}^{k-1}(f_1^k(x))\in f_{k+1}^{k-1}(U_1)\subseteq (U_1)_{k-1},\\
	&\dots\\
	f_1^{(n-1)k}(x)&\in U_{n-1},\dots,f_1^{nk-1}(x)=f_{(n-1)k+1}^{k-1}(f_1^{(n-1)k}(x))\in f_{(n-1)k+1}^{k-1}(U_{n-1})\subseteq (U_{n-1})_{k-1}.
\end{align*}
This means that $x\in X_{f_{1,\infty}}(\mathbf{V})$. Therefore, $X_{f_{1,\infty}^k}(\mathbf{U})\subseteq X_{f_{1,\infty}}(\mathbf{V})$. We denote the set of all such $\mathbf{V}$ by $\mathcal{G}_{f_{1,\infty}}$, then $\mathcal{G}_{f_{1,\infty}}$ also covers $K$ and $\mathcal{G}_{f_{1,\infty}}\subseteq \bigcup_{n\geq kN}S_{n,f_{1,\infty}}(\mathcal{U}')$.

Therefore 
\begin{align}\label{equineq}
M(f_{1,\infty}^k,K,s,\mathcal{U},\alpha,N)&=\inf_{\mathcal{G}_{f_{1,\infty}^k}}\left\{ \sum_{\mathbf{U}\in\mathcal{G}_{f_{1,\infty}^k}}e^{-\alpha\cdot m(\mathbf{U})^s} \right\}\nonumber\\
&\geq \inf_{\mathcal{G}_{f_{1,\infty}}}\left\{\sum_{\mathbf{V}\in \mathcal{G}_{f_{1,\infty}}}e^{-\alpha\cdot\frac{m(\mathbf{V})^s}{k^s}} \right\}\nonumber\\
&=M(f_{1,\infty},K,s,\mathcal{U}',\frac{\alpha}{k^s},kN).
\end{align}
Similar to the proof of ~\Cref{T:inequality}, then $m(f_{1,\infty}^k,K,s,\mathcal{U},\alpha)\geq m(f_{1,\infty},K,s,\mathcal{U}',\alpha/k^s)$. Hence 
\[
k^s\cdot D(f_{1,\infty},K,s,\mathcal{U}')\leq D(f_{1,\infty}^k,K,s,\mathcal{U}),
\] 
and $D(f_{1,\infty}^k,K,s)\geq k^s D(f_{1,\infty},K,s)$ by taking the supremum of all finite open cover $\mathcal{U}$ of $X$. Following a similar proof of ~\Cref{T:power}, we have $D(f_{1,\infty}^k,K)\geq D(f_{1,\infty},K)$. Now combining with ~\Cref{T:power}, then $D(f_{1,\infty}^k,K)= D(f_{1,\infty},K)$.
\end{proof}



\begin{theorem}\label{T:periodic}
Let $(X,f_{1,\infty})$ be an NDS, and $K$ be a nonempty subset of $X$. If the sequence of maps $f_{1,\infty}$ has period $k$. Then 
\[
D(f_{1,\infty}^k,K)= D(f_{1,\infty},K).
\]
\end{theorem}

\begin{proof}
By ~\Cref{T:power}, we need only to prove $D(f_{1,\infty}^k,K,s)\geq k^s D(f_{1,\infty},K,s)$ holds for any $s>0$. Let $\mathbf{V}=(V_{i_0},V_{i_1},\dots,V_{i_{n-1}})\in \mathcal{G}_{f_{1,\infty}^k}$ with length $m(\mathbf{V})=n$ and $K\subseteq \bigcup_{\mathbf{V}\in \mathcal{G}_{f_{1,\infty}^k}} X_{f_{1,\infty}^k}(\mathbf{V})$. For any open cover $\mathcal{U}$ of $X$, then $\mathcal{V}:=\bigvee_{j=0}^{k-1}f_1^{-j}(\mathcal{U})$ is also an open cover of $X$. For any $U_{i_j}\in\mathcal{U}$, we put
\begin{align}\label{vit}
V_{i_t}=U_{i_{tk}}\cap f_1^{-1}(U_{i_{tk+1}})\cap\dots\cap f_1^{-(k-1)}(U_{i_{(t+1)k-1}}),\quad 0\leq t\leq n-1.
\end{align}
It is clear that $\mathcal{G}_{f_{1,\infty}^k}\subseteq \bigcup_{n\geq N}S_{N,f_{1,\infty}^k}(\mathcal{U})$. Now define 
\[\mathbf{U}=(U_{i_0},U_{i_1},\dots,U_{i_{k-1}},U_{i_k},\dots,U_{i_{2k-1}},\dots,U_{i_{(n-1)k}},\dots,U_{i_{nk-1}}).\]
Then 
\begin{equation}\label{kmv}
m(\mathbf{U})=kn=k\cdot m(\mathbf{V}),
\end{equation}
and $\mathbf{U}\in \bigcup_{n\geq kN}S_{n,f_{1,\infty}}(\mathcal{U})$. We denote the set of all such $\mathbf{U}$ by $\mathcal{G}_{f_{1,\infty}}$.

Notice that
\begin{align*}
X_{f_{1,\infty}}(\mathbf{U})&=\{x\in X: f_1^j(x)\in U_{i_j}, 0\leq j\leq nk-1\}\\
	&=\left[U_{i_0}\cap f_1^{-1}(U_{i_1})\cap\dots\cap f_1^{-(k-1)}(U_{i_{k-1}}) \right]\\
	&\quad \bigcap \left[ f_1^{-k}(U_{i_k})\cap f_1^{-(k+1)}(U_{i_{k+1}})\cap \dots\cap f_1^{-(2k-1)}(U_{i_{2k-1}}) \right]\bigcap\dots \\
	&\quad \bigcap \left[f_1^{-(n-1)k}(U_{i_{(n-1)k}})\cap\dots\cap f_1^{-(nk-1)}(U_{i_{nk-1}}) \right]\\
	&=\left[U_{i_0}\cap f_1^{-1}(U_{i_1})\cap\dots\cap f_1^{-(k-1)}(U_{i_{k-1}}) \right]\\
	&\quad \bigcap f_1^{-k}\left(U_{i_k}\cap f_1^{-1}(U_{i_{k+1}})\cap\dots\cap f_1^{-(k-1)}(U_{i_{2k-1}}) \right)\bigcap\dots\\
	&\quad \bigcap f_1^{-(n-1)k}\left(U_{i_{(n-1)k}}\cap f_1^{-1}(U_{i_{(n-1)k}+1})\cap\dots\cap f_1^{-(k-1)}(U_{i_{nk-1}})\right)\quad (\text{By} ~\eqref{reverse})\\
	&=V_{i_0}\cap f_1^{-k}(V_{i_1})\cap\dots\cap f_1^{-(n-1)k}(V_{i_{n-1}})\quad (\text{By} ~\eqref{vit})\\
	&=X_{f_{1,\infty}^k}(\mathbf{V}).
\end{align*}
Then the collection $\mathcal{G}_{f_{1,\infty}}$ also covers $K$. It follows that 
\begin{align*}
	M(f_{1,\infty}^k,K,s,\mathcal{V},\alpha,N)&=\inf_{\mathcal{G}_{f_{1,\infty}^k}}\left\{ \sum_{\mathbf{V}\in\mathcal{G}_{f_{1,\infty}^k}}e^{-\alpha\cdot m(\mathbf{V})^s} \right\}\\
	&\geq \inf_{\mathcal{G}_{f_{1,\infty}}}\left\{\sum_{\mathbf{U}\in \mathcal{G}_{f_{1,\infty}}}e^{-\frac{\alpha}{k^s}m(\mathbf{U})^s } \right\}\quad (\text{By} ~\eqref{kmv})\\
	&\geq M(f_{1,\infty},K,s,\mathcal{U},\frac{\alpha}{k^s},kN).
\end{align*}
Notice that the last inequality is similar to ~\eqref{equineq}, so following a similar discussion then we complete the proof.
\end{proof}
	
\begin{corollary}
Let $(X,d)$ be a compact metric space, and $f,g$ be two continuous selfmaps of $X$. Then for any $f,g$-invariant subset $K\subseteq X$, 
\[
D(f\circ g,K)=D(g\circ f,K).
\] 
\end{corollary}

\begin{proof}
By ~\Cref{T:inequality}, we have
\[
D(\{f,g,f,g,\dots\},K)\leq D(\{g,f,g,f,\dots\},K)\leq D(\{f,g,f,g,\dots\},K).
\]
This implies that 
\[
D(\{f,g,f,g,\dots\},K)=D(\{g,f,g,f,\dots\},K).
\] 
Then by ~\Cref{T:periodic}, 
\begin{align*}
D(f\circ g,K)&=D(\{f\circ g,f\circ g,\dots\},K)\\
	&=D(\{f,g,f,g,\dots\},K)\\
	&=D(\{g,f,g,f,\dots\},K)\\
	&=D(\{g\circ f,g\circ f,\dots\},K)\\
	&=D(g\circ f,K).
\end{align*}
\end{proof}

\begin{definition}\cite{KS1996}
Let $(X,d)$ and $(Y,\rho)$ be two compact metric spaces, $f_{1,\infty}$ be a sequence of continuous selfmaps of $X$, and $g_{1,\infty}$ be a sequence of continuous selfmaps of $Y$. Then we have two NDSs $(X,f_{1,\infty})$ and $(Y,g_{1,\infty})$. Suppose that $\pi_{1,\infty}$ is a sequence of equicontinuous surjective maps from $X$ to $Y$. If $\pi_{i+1}\circ f_i=g_i\circ \pi_i$ for all $i\geq 1$, then we say $\pi_{1,\infty}$ is a \emph{topologically equisemiconjugacy} between $f_{1,\infty}$ and $g_{1,\infty}$ and $(X,f_{1,\infty})$ is equisemiconjugates with $(Y,g_{1,\infty})$. 
If $\pi_{1,\infty}$ is further an equicontinuous sequence of homeomorphisms such that $\pi^{-1}_{1,\infty}=\{\pi_i^{-1}\}_{i=1}^{\infty}$ of inverse homeomorphisms is also equicontinuous, then we say $\pi_{1,\infty}$ \emph{topologically equiconjugates $(X,f_{1,\infty})$ and $(Y,g_{1,\infty})$}.
\end{definition}

Kolyada and Snoha ~\cite{KS1996} prove that the topological entropy of sequence of continuous maps is invariant up to equiconjugacy. Li et al. ~\cite{LZW2019} show that the classical upper entropy dimension $\overline{D}_K(f_{1,\infty})$ is invariant under equiconjugacy. The following ~\Cref{equisemiconjuates} further shows that the Pesin $s$-topological entropy $D(f_{1,\infty},X,s)$ and the Pesin topological entropy dimension $D(f_{1,\infty},X)$ are also invariant under equiconjugacy.

\begin{theorem}\label{equisemiconjuates}
Let $(X, d)$ and $(Y,\rho)$ be two compact metric spaces, $f_{1,\infty}$ be a sequence of continuous selfmaps of $X$, $g_{1,\infty}$ be a sequence of continuous selfmaps of $Y$, and $K$ be a nonempty subset of $X$. If $(X,f_{1,\infty})$ is equisemiconjuates with $(Y,g_{1,\infty})$, then
\[
D(f_{1,\infty},K)\geq D(g_{1,\infty},\pi_1(K)).
\]
If $(X,f_{1,\infty})$ is equiconjugates with $(Y,g_{1,\infty})$, then we further have 
\[
D(f_{1,\infty},K)= D(g_{1,\infty},\pi_1(K)).
\] 
Particularly, $D(f_{1,\infty},X)$ is invariant up to equiconjugacy, i.e., $D(f_{1,\infty},X)= D(g_{1,\infty},Y)$.
\end{theorem}

\begin{proof}
To prove the first inequality, we need only to show that for any $s>0$, we have $D(f_{1,\infty},K,s)\geq D(g_{1,\infty},\pi_1(K),s)$. In other words, we need to find some appropriate $\mathcal{U},\mathcal{V},s,\alpha$ and $N$ such that $M(g_{1,\infty},\pi_1(K),s,\mathcal{U},\alpha,N)\leq M(f_{1,\infty},K,s,\mathcal{V},\alpha,N)$. 

Note that $\pi_{1,\infty}$ is topologically equisemiconjuates $f_{1,\infty}$ and $g_{1,\infty}$, then $\pi_{j+1}\circ f_j=g_j\circ \pi_j$.
\[
\begin{tikzcd}
X \arrow{r}{f_1} \arrow{d}{\pi_1} & X \arrow{r}{\dots}\arrow{d}{\pi_2} & X\arrow{d}{\pi_j} \arrow{r}{f_j}&X \arrow{d}{\pi_{j+1}} \\
Y \arrow[swap]{r}{g_1}& Y\arrow[swap]{r}{\dots}& Y\arrow[swap]{r}{g_j}&Y
\end{tikzcd}
\]
By the commutative diagram, 
\begin{align}\label{fpi}
\pi_{j+1}\circ f_1^j=g_1^j\circ \pi_1, \quad \pi_1^{-1}\circ g_1^{-j}=f_1^{-j}\circ \pi_{j+1}^{-1},\qquad \forall j\geq 1.
\end{align}

Let $\mathcal{U}$ be a finite open cover of $Y$. For any $\varepsilon>0$ and $U\in\mathcal{U}$, set
\[
\mathcal{G}_{g_{1,\infty}}:=\{\mathbf{U}=(U_{i_0},U_{i_1},\dots,U_{i_{n-1}}): \rho(U_{i_j})<\varepsilon, U_{i_j}\in\mathcal{U}, 0\leq j\leq n-1\}.
\]
By the definition of $M(g_{1,\infty},\pi_1(K),s,\mathcal{U},\alpha,N)$, for any $\mathbf{U}=(U_{i_0},U_{i_1},\dots,U_{i_{n-1}})\in \mathcal{G}_{g_{1,\infty}}$, $m(\mathbf{U})\geq N$ and $\mathcal{G}_{g_{1,\infty}}$ covers $\pi_1(K)$, then
\begin{align*}
\pi_1(K)&\subseteq \bigcup_{\mathbf{U}\in\mathcal{G}_{g_{1,\infty}}}Y_{g_{1,\infty}}(\mathbf{U})\\
	&=\bigcup_{\mathbf{U}\in\mathcal{G}_{g_{1,\infty}}}\left\{y\in Y: g_1^j(y)\in U_{i_j},0\leq j\leq n-1\right\}\\
	&=\bigcup_{\mathbf{U}\in\mathcal{G}_{g_{1,\infty}}}\left(U_{i_0}\cap g_1^{-1}(U_{i_1})\cap\dots\cap g_1^{-(n-1)}(U_{i_{n-1}}) \right).
\end{align*}
Thus 
\begin{align*}
K&\subseteq \bigcup_{\mathbf{U}\in\mathcal{G}_{g_{1,\infty}}} \pi_1^{-1}\left(U_{i_0}\cap g_1^{-1}(U_{i_1})\cap\dots\cap g_1^{-(n-1)}(U_{i_{n-1}}) \right)\\
	&=\bigcup_{\mathbf{U}\in\mathcal{G}_{g_{1,\infty}}} \pi_1^{-1}(U_{i_0})\cap \pi_1^{-1}(g_1^{-1}(U_{i_1}))\cap\dots\cap \pi_1^{-1}(g_1^{-(n-1)}(U_{i_{n-1}}))\\
	&=\bigcup_{\mathbf{U}\in\mathcal{G}_{g_{1,\infty}}} \pi_1^{-1}(U_{i_0})\cap f_1^{-1}(\pi_2^{-1}(U_{i_1}))\cap\dots\cap f_1^{-(n-1)}(\pi_{n}^{-1}(U_{i_{n-1}}))\quad (\text{by} ~\eqref{fpi}).
\end{align*}

For $\mathbf{U}=(U_{i_0},U_{i_1},\dots,U_{i_{n-1}})\in \mathcal{G}_{g_{1,\infty}}$, we choose $U_{i_j}\in\mathbf{U}$ with $d(\pi_{j+1}^{-1}(U_{i_j}))<\delta$ for $0\leq j\leq n-1$. Then for such collection of $U_{i_j}$, let $V_{i_j}:=\pi_{j+1}^{-1}(U_{i_j})\in\mathcal{V}$ and $\mathbf{V}:=(V_{i_0},V_{i_1},\dots,V_{i_{n-1}})\in \mathcal{G}_{f_{1,\infty}}$. Then $m(\mathbf{U})=m(\mathbf{V})$. Note that $\pi_{1,\infty}$ is equicontinuous, then $d(V_{i_j})<\delta$ implies that $\rho(\pi_{j+1}(V_{i_j})=\rho\left(\pi_{j+1}(\pi_{j+1}^{-1}(U_{i_j}))\right)=\rho(U_{i_j})<\varepsilon$.

Therefore, we further have
\begin{align*}
K&\subseteq \bigcup_{\mathbf{U}\in\mathcal{G}_{g_{1,\infty}}} V_{i_0}\cap f_1^{-1}(V_{i_1})\cap\dots\cap f_1^{-(n-1)}(V_{i_{n-1}})\\
	&\subseteq \bigcup_{\mathbf{V}\in\mathcal{G}_{f_{1,\infty}}} X_{f_{1,\infty}}(\mathbf{V}).
\end{align*}
Hence $\mathcal{G}_{f_{1,\infty}}$ covers $K$ and 
\begin{align*}
M(g_{1,\infty},\pi_1(K),s,\mathcal{U},\alpha,N)&=\inf_{\mathcal{G}_{g_{1,\infty}}}\left\{\sum_{\mathbf{U}\in \mathcal{G}_{g_{1,\infty}}}e^{-\alpha m(\mathbf{U})^s} \right\}\\
	&\leq \inf_{\mathcal{G}_{f_{1,\infty}}}\left\{\sum_{\mathbf{V}\in \mathcal{G}_{f_{1,\infty}}}e^{-\alpha m(\mathbf{V})^s} \right\}\\
	&=M(f_{1,\infty},K,s,\mathcal{V},\alpha,N).
\end{align*}
A similar discussion implies that $D(f_{1,\infty},K,s)\geq D(g_{1,\infty},\pi_1(K),s)$ and therefore $D(f_{1,\infty},K)\geq D(g_{1,\infty},\pi_1(K))$ follows.

In the case $\pi_{1,\infty}$ is homeomorphism, then $\pi_{j+1}^{-1}\circ g_j=f_j\circ \pi_{j}^{-1}$, it follows that 
\[
D(g_{1,\infty},\pi_1(K),s)\geq D(f_{1,\infty},\pi_1^{-1}(\pi_1(K)),s)=D(f_{1,\infty},K,s).
\] 
Thus we get $D(f_{1,\infty},K,s)= D(g_{1,\infty},\pi_1(K),s)$, then $D(f_{1,\infty},K)= D(g_{1,\infty},\pi_1(K))$. 

In particular when $K=X$, since $\pi_i$ is surjective for all $i$, then $\pi_1(X)=Y$. It follows that $D(f_{1,\infty},X,s)= D(g_{1,\infty},Y,s)$ and $D(f_{1,\infty},X)= D(g_{1,\infty},Y)$.
\end{proof}

\section{Relationships among topological entropy dimensions}
\begin{lemma}\label{l:Nrn}
Let $(X,d)$ be a compact metric space, $(X,f_{1,\infty})$ be an NDS, and $K$ be a nonempty $f_{1,\infty}$-invariant subset of $X$. If $\mathcal{U}$ is an open cover of $X$ with Lebesgue number $\delta$, then 
\[
\mathcal{N}(\{X(\mathbf{U}):\mathbf{U}\in\mathcal{G}\})\leq s_n(f_{1,\infty},K,\delta/2),
\]
where the collection of strings $\mathcal{G}$ covers $K$.
\end{lemma}

\begin{proof}
The proof is a variation of ~\cite[Theorem 7.7]{Wal2000a}. If $\mathcal{U}$ is an open cover of $X$ with Lebesgue number $\delta$, we let $E\subseteq X$ be an $(n,\delta/2)$-spanning set of $K$. Then for any $y\in K$, there exists $x\in E$ such that $d_n(x,y)\leq \frac{\delta}{2}$, i.e., $y\in f_1^{-j}(\overline{B}(f_1^j(x),\frac{\delta}{2}))$ for $1\leq j\leq n-1$, or equivalently, 
\[
K\subseteq \bigcup_{x\in E}\bigcap_{j=0}^{n-1}f_1^{-j}\left(\overline{B}(f_1^j(x),\frac{\delta}{2}\right).
\]
Note that $K$ is $f_{1,\infty}$-invariant, then for each $0\leq j\leq n-1$, 
\begin{align*}
f_1^{-j}\left(\overline{B}(f_1^j(x),\frac{\delta}{2}\right)&=f_1^{-j}\left(\overline{B}(f_1^j(x),\frac{\delta}{2})\right)\cap K \\
&=f_1^{-j}\left(\overline{B}(f_1^j(x),\frac{\delta}{2})\right)\cap f_1^{-j}(K)\\
&=f_1^{-j}\left(\overline{B}(f_1^j(x),\frac{\delta}{2})\cap K \right)
\end{align*}
Since for each $j$, $\overline{B}(f_1^j(x),\delta/2)\cap K$ is a subset of a member of $\mathcal{U}|_K$. Then given that $\mathcal{G}$ covers $K$, we have $\mathcal{N}(\{X(\mathbf{U}):\mathbf{U}\in\mathcal{G}\})\leq r_n(f_{1,\infty},K,\delta/2)\leq s_n(f_{1,\infty},K,\delta/2)$.
\end{proof}

\begin{theorem}\label{compare}
Let $(X,d)$ be a compact metric space, $(X,f_{1,\infty})$ be an NDS and $K$ be a nonempty $f_{1,\infty}$-invariant subset of $X$. Then
\[
D(f_{1,\infty},K)\leq \underline{D}_K(f_{1,\infty})\leq\overline{D}_K(f_{1,\infty}).
\]
\end{theorem}

\begin{proof}
The second inequality is obvious, so we need only to prove $D(f_{1,\infty},X,s)\leq \underline{h}_X(f_{1,\infty},s)$ for any $s>0$. Let $\mathcal{U}$ be a finite open cover of $X$ with Lebesgue number $\delta$. Note that by ~\eqref{M}, for any $s>0$ and $\alpha\in \mathbb{R}$,
\[
M(f_{1,\infty},K,s,\mathcal{U},\alpha,N)=\inf\left\{\sum\limits_{\mathbf{U}\in \mathcal{G} } e^{-\alpha m(\mathbf{U})^s}\right\},
\]
where the infimum is taken over all finite or countable collections of strings $\mathcal{G}\subseteq S(\mathcal{U})$ such that $m(\mathbf{U})\geq N$ for all $\mathbf{U}\in\mathcal{G}$ and $\mathcal{G}$ covers $K$. 

Since $K$ is $f_{1,\infty}$-invariant, then by ~\Cref{l:Nrn}, 
\begin{align*}
M(f_{1,\infty},K,s,\mathcal{U},\alpha,N)&\leq e^{-\alpha N^s}\cdot \mathcal{N}(\{X(\mathbf{U}):\mathbf{U}\in\mathcal{G}\})\\
&\leq e^{-\alpha N^s}\cdot s_N(f_{1,\infty},K,\delta/2).
\end{align*}
Now letting $N\to\infty$, it follows that
\[
m(f_{1,\infty},K,s,\mathcal{U},\alpha)\leq \mathop{\lim\inf}_{N\to\infty}\left(e^{-\alpha+\frac{1}{N^s}\log s_N(f_{1,\infty},K,\delta/2)}\right)^{N^s}.
\]
By ~\Cref{l:hd},
\[
\underline{h}_K(f_{1,\infty},s)=\underline{h}_{span}(f_{1,\infty},K,s)=\lim_{\varepsilon\to0}\mathop{\lim\inf}_{N\to\infty}\frac{1}{N^s}\log s_N(f_{1,\infty},K,\varepsilon).
\]
If $\alpha>\mathop{\lim\inf}_{N\to\infty}\frac{1}{N^s}\log s_N(f_{1,\infty},K,\delta/2)$, then $m(f_{1,\infty},K,s,\mathcal{U},\alpha)=0$.
Therefore, $D(f_{1,\infty},K,s,\mathcal{U}) \leq \mathop{\lim\inf}_{N\to\infty}\frac{1}{N^s}\log s_N(f_{1,\infty},K,\delta/2)$. Letting $|\mathcal{U}|\to 0$ then by ~\Cref{P:distance}, we get $D(f_{1,\infty},K,s)\leq \underline{h}_K(f_{1,\infty},s)$ and $D(f_{1,\infty},K)\leq \underline{D}_K(f_{1,\infty})$.
\end{proof}

\begin{corollary}\label{corollaryX}
Let $(X,d)$ be a compact metric space, and $(X,f_{1,\infty})$ be an NDS. Then	
\[
D(f_{1,\infty},X)\leq \underline{D}_X(f_{1,\infty})\leq \overline{D}_X(f_{1,\infty}).
\]
\end{corollary}

\begin{theorem}\label{T:htopD}
Let $(X,f_{1,\infty})$ be an NDS, and $K$ be a nonempty subset of $X$.
\begin{enumerate}
\item If $h_{top}(f_{1,\infty},K)<+\infty$, then $0\leq \overline{D}_K(f_{1,\infty})\leq 1;$

\item if $0\leq \overline{D}_K(f_{1,\infty})<1$, then $h_{top}(f_{1,\infty},K)<+\infty;$

\item if $D(f_{1,\infty},K,1)<+\infty$, then $0\leq D(f_{1,\infty},K)\leq 1;$

\item if $0\leq D(f_{1,\infty},K)<1$, then $D(f_{1,\infty},K,1) <+\infty$.
\end{enumerate}
\end{theorem}

\begin{proof}
(1) For any finite open cover $\mathcal{U}$ of $X$, since 
\[
h_{top}(f_{1,\infty},K,\mathcal{U})=\mathop{\lim\sup}_{n\to\infty}\frac{1}{n}\log\mathcal{N}(\mathcal{U}_1^n|_K)\leq h_{top}(f_{1,\infty},X,\mathcal{U})
\] 
and
\begin{align}
\overline{h}_K(f_{1,\infty},s,\mathcal{U})&=\mathop{\lim\sup}_{n\to\infty}\frac{1}{n^s}\log\mathcal{N}(\mathcal{U}_1^n|_K) \nonumber\\
&=\mathop{\lim\sup}_{n\to\infty}\frac{1}{n^{s-1}}\cdot\left[\frac{1}{n}\log\mathcal{N}(\mathcal{U}_1^n|_K)\right].  \label{olh}
\end{align}
Thus if $h_{top}(f_{1,\infty},K,\mathcal{U})<+\infty$ and $s>1$, then $\overline{h}_K(f_{1,\infty},s,\mathcal{U})\leq \mathop{\lim\sup}_{n\to\infty}\frac{1}{n^{s-1}}\cdot \mathop{\lim\sup}_{n\to\infty}\frac{1}{n}\log \mathcal{N}(\mathcal{U}_1^n|_K)=0$, which immediately implies $\overline{D}_K(f_{1,\infty})\leq 1$.

(2) If $h_{top}(f_{1,\infty},K)=+\infty$, for each $s < 1$, by \eqref{olh}, then 
\[\overline{h}_K(f_{1,\infty},s,\mathcal{U})=\mathop{\lim\sup}_{n\to\infty}n^{s-1}\cdot\left[\frac{1}{n}\log\mathcal{N}(\mathcal{U}_1^n|_K)\right]=+\infty
\] 
This implies that $\overline{D}_K(f_{1,\infty})\geq 1$. Equivalently, $0\leq \overline{D}_K(f_{1,\infty})<1$ implies $0\leq h_{top}(f_{1,\infty},K)<+\infty$.

(3) By ~\eqref{M}, for any $\alpha\in \mathbb{R}$ and $N>0$,
\begin{align}
\sum_{\mathbf{U}\in \mathcal{G}} e^{-\alpha m(\mathbf{U})^s}&=\sum_{\mathbf{U}\in \mathcal{G}} e^{-\alpha [m(\mathbf{U})\cdot m(\mathbf{U})^{s-1}]}\nonumber \\ 
&=\sum_{\mathbf{U}\in \mathcal{G}} e^{-\alpha m(\mathbf{U})}\cdot e^{-\alpha m(\mathbf{U})^{s-1}}\label{leqM}.
\end{align}
If $D(f_{1,\infty},K,1)<+\infty$ and $s>1$, then 
\[
\sum_{\mathbf{U}\in \mathcal{G}} e^{-\alpha m(\mathbf{U})^s} \leq \sum_{\mathbf{U}\in \mathcal{G}} e^{-\alpha m(\mathbf{U})}\cdot \sum_{\mathbf{U}\in \mathcal{G}} e^{-\alpha m(\mathbf{U})^{s-1}}
\] 
Taking the infimum over all finite or countable collections of strings $\mathcal{G}$ on both sides such that $m(\mathbf{U})\geq N$ for all $\mathbf{U}\in \mathcal{G}$ and $\mathcal{G}$ covers $K$, then
\[
M(f_{1,\infty},K,s,\mathcal{U},\alpha,N)\leq M(f_{1,\infty},K,1,\mathcal{U},\alpha,N)\cdot M(f_{1,\infty},K,s-1,\mathcal{U},\alpha,N).
\] 
Now by the definition of Pesin entropy dimension, 
\[
D(f_{1,\infty},K,s)\leq D(f_{1,\infty},K,1)\cdot D(f_{1,\infty},K,s-1).
\]
Thus $D(f_{1,\infty},K,1)<+\infty$ implies $D(f_{1,\infty},K,s)<+\infty$. This means that $D(f_{1,\infty},K,s)=0$ and in this case we have $s>D(f_{1,\infty},K)$. Hence $D(f_{1,\infty},K)\leq 1$.

(4) If $D(f_{1,\infty},K,1)=Ph(f_{1,\infty},K)=+\infty$ and $s<1$, by the fact that $e^{-\alpha m(\mathbf{U})^{s-1}}>0$, and taking the infimum over all finite or countable collections of strings $\mathcal{G}$ on both sides of ~\eqref{leqM} such that $m(\mathbf{U})\geq N$ for all $\mathbf{U}\in \mathcal{G}$ and $\mathcal{G}$ covers $K$, then following a similar discussion we can check that $D(f_{1,\infty},K,s)=+\infty$. This implies that $D(f_{1,\infty},K)\geq 1$. Equivalently, we have $0\leq D(f_{1,\infty},K)<1$ implies $D(f_{1,\infty},K,1)<+\infty$.
\end{proof}

\begin{remark}
(1) and (2) are not equivalent (even for $K=X$), since there are examples shows that the topological entropy dimension could be $1$ but the topological entropy be zero, see ~\cite[Theorem 3.6]{CL2010} for more details. Similarly, (3) and (4) are also not equivalent. 
\end{remark}

\begin{lemma} \label{pesinentropy}
Let $(X,d)$ be a compact metric space, and $(X,f_{1,\infty})$ be an NDS. Then for any subset $K\subseteq X$,
\[
0\leq D(f_{1,\infty},K,1)=Ph(f_{1,\infty},K)\leq h_{top}(f_{1,\infty},K)\leq h_{top}(f_{1,\infty}).
\]
\end{lemma}

\begin{proof}
The proof follows directly from ~\cite[Theorem 3.1]{li2015} and a small modification of ~\cite[Theorem 3.10]{LY2023}, so we omit it.
\end{proof}

\begin{remark}
The condition in (1) and (3) of ~\Cref{T:htopD} could be strengthen to $h_{top}(f_{1,\infty},X)<+\infty$. Rodrigues and Acevedo ~\cite{RA2022} introduced the metric mean dimension for nonautonomous dynamical systems, they showed that if $h_{top}(f_{1,\infty},X)$ is finite then its metric mean dimension $\textnormal{mdim}_M(f_{1,\infty},X)=0$. So by \Cref{pesinentropy}, we may strengthen (1) and (3) of \Cref{T:htopD} as follows: If $\textnormal{mdim}_M(f_{1,\infty},X)=0$, then $0\leq D(f_{1,\infty},K)\leq \overline{D}_K(f_{1,\infty})\leq 1$.
\end{remark}

\begin{theorem}\label{coincide}
Let $(X,f_{1,\infty})$ be an NDS, and $K$ be a nonempty subset of $X$. 
\begin{enumerate}
\item If $0< D(f_{1,\infty},K,1)\leq h_{top}(f_{1,\infty},K)<+\infty$ and $K$ is $f_{1,\infty}$-invariant, then 
\[
D(f_{1,\infty},K)=D_K(f_{1,\infty})=1;
\]

\item if $0< D(f_{1,\infty},K,1)\leq h_{top}(f_{1,\infty},X)<+\infty$, then 
\[
D(f_{1,\infty},X)=D_X(f_{1,\infty})=1.
\]  
\end{enumerate}
\end{theorem}

\begin{proof}
(1) If $0< D(f_{1,\infty},K,1)\leq h_{top}(f_{1,\infty},K)<+\infty$, then by ~\Cref{compare}, ~\Cref{T:htopD} and ~\Cref{pesinentropy}, 
\[
0\leq D(f_{1,\infty},K)\leq \underline{D}_K(f_{1,\infty}) \leq \overline{D}_K(f_{1,\infty})\leq 1.
\] 
Note that $D(f_{1,\infty},K,1)>0$ and ${D}(f_{1,\infty},K,s)$ jumps from $+\infty$ to $0$ at the unique critical value, so it must corresponds to $s=1$. Thus $D(f_{1,\infty},K)=1$, and it follows that $D(f_{1,\infty},K)=\underline{D}_K(f_{1,\infty})=\overline{D}_K(f_{1,\infty})=1$.

(2) Since $X$ is naturally $f_{1,\infty}$-invariant, thus this result follows from (1).
\end{proof}

\begin{corollary}\label{corofinite}
Let $(X,f_{1,\infty})$ be an NDS, and $K$ be a nonempty $f_{1,\infty}$-invariant subset of $X$. 
\begin{enumerate}
	\item If $0< D(f_{1,\infty},K,1)<+\infty$, then $D(f_{1,\infty},K)=1;$
	\item if $0<h_{top}(f_{1,\infty},K)<+\infty$, then $D(f_{1,\infty},K)=1$.
\end{enumerate}
\end{corollary}

The results of this study, as indicated by \Cref{T:htopD}, \Cref{coincide} and \Cref{corofinite}, suggest that both the classical topological entropy dimension and the Pesin topological entropy dimension should be primarily employed to distinguish the zero topological entropy systems. In the remainder of this section, several examples will be presented where the topological entropy is zero, yet the topological entropy dimensions vary. For the sake of convenience, the discussion will be limited on the whole set $X$. It is noteworthy that, by ~\cite[Theorem 3.6]{CL2010}, even the topological entropy is zero, the topological entropy dimension may still be full. 

\begin{example}\label{example1}
Let the phase space $X$ be finite and $f_{1,\infty}$ be a sequence of continuous selfmaps on $X$. Then $h_{top}(f_{1,\infty})=D(f_{1,\infty},X,1)=0$ and $D(f_{1,\infty},X) = D_X(f_{1,\infty})=0$. 

This assertion can be verified as follows: Since the minimal cardinality of any $(n,\varepsilon)$-spanning set of $X$ is smaller than the cardinality of $X$, then $r_n(f_{1,\infty},X,\varepsilon)<+\infty$. It is clear that $h_{top}(f_{1,\infty})=0$. And by \Cref{pesinentropy}, then $D(f_{1,\infty},X,1)=Ph(f_{1,\infty},X)=0$. By \Cref{l:hd} and \Cref{pesinentropy}, then for any $0 \leq s \leq 1$, we also have $\overline{h}_X(f_{1,\infty},s)=0$, and this further implies that $\overline{D}_X(f_{1,\infty})=0$. By \Cref{corollaryX}, it follows that $D(f_{1,\infty},X) = \underline{D}_X(f_{1,\infty})=\overline{D}_X(f_{1,\infty})=0$.
\end{example}

\begin{example}
Let $X$ be a compact topological space, and each $f_i\in f_{1,\infty}$ be continuous contractive selfmaps on $X$. Then $h_{top}(f_{1,\infty})=D(f_{1,\infty},X,1)=0$ and $D(f_{1,\infty},X) = D_X(f_{1,\infty})=0$. 

This assertion can be verified as follows: Let $\mathcal{U}$ be an open cover of $X$ with Lebesgue number $\delta$. By \cite[Theorem 7.7]{Wal2000a}, then $\mathcal{N}\left (\bigvee_{j=0}^{n-1}f_1^{-j}(\mathcal{U}) \right)\leq r_n(\delta/2,X)$. Note that each $f_i$ is contractive, thus $r_n(\delta/2,X)$ decrease as $n \to \infty$ and has a finite upper bound. Similar to \Cref{example1}, we can check that $h_{top}(f_{1,\infty})=0$, and $D(f_{1,\infty},K,1)=Ph(f_{1,\infty},X)=0$. Similarly, we also have $D(f_{1,\infty},X) = \underline{D}_X(f_{1,\infty})=\overline{D}_X(f_{1,\infty})=0$.
\end{example}

\begin{example} Let $X=S^1$, and the sequence of continuous selfmaps $f_{1,\infty}$ on $X$ be defined as follows: Fix $m\in \mathbb{N}$ and $s\in (0,1)$, for any $k \geq 1$, let
\[
f_{k}(x)= \begin{cases}
 mx \pmod 1\quad & \text{if } [k^s]>[(k-1)^s],\\ 
\textnormal{id}_X \quad & \text{if }[k^s]=[(k-1)^s],
 \end{cases}
\] 
where $[k^s]$ denotes the integer part of $k^s$. Then $D(f_{1,\infty},X,1)=Ph(f_{1,\infty},X)=h_{top}(f_{1,\infty})=0$, while $D_{X}(f_{1,\infty})=s$. 

This example refers to \cite[Example 4.1]{LZW2019} and \cite[Example 3.7]{YWW2020} respectively. The assertion can be verified as follows: According to the aforementioned references, the topological entropy of $(X,f_{1,\infty})$ is $h_{top}(f_{1,\infty})=0$ and the topological entropy dimension is $D_{X}(f_{1,\infty})=s$. By \Cref{pesinentropy} and \Cref{corollaryX}, then we deduce that $D(f_{1,\infty},X,1)=Ph(f_{1,\infty},X)=0$ and $0\leq D(f_{1,\infty},X)\leq D_X(f_{1,\infty})=s$. To the best of our knowledge, there exists no general method to compute the exact value of Pesin topological entropy dimension. However, By the definition of entropy dimension, without much hard, the Pesin topological entropy dimension $D(f_{1,\infty},X)>0$. Consequently, we can conclude that the Pesin topological entropy dimension is contained within the interval $(0,s]$.
\end{example}

\section*{Acknowledgements}
The author would like to thank Professor Li Jian for his kindly suggestions, and also thank the respectful reviewers for their valuable comments for improvement of the article. This work was supported by the National Natural Science Foundation of China (No. 12271185) and the Basic and Applied Basic Research Foundation of Guangdong Province (No. 2022A1515011124 and No. 2023A1515110084). 

\bibliographystyle{siam} 
\bibliography{refs}
\end{document}